\documentclass[11pt]{amsart}
\usepackage{amsmath}
\usepackage{eurosym}
\usepackage{amsfonts}

\newtheorem{theorem}{Theorem}
\theoremstyle{plain}

\newtheorem{corollary}{Corollary}

\newtheorem{example}{Example}

\newtheorem{lemma}{Lemma}

\newtheorem{proposition}{Proposition}
\newtheorem{remark}{Remark}

\numberwithin{equation}{section}
\input{tcilatex}

\begin{document}
\title[integro-differential equations]{On the Cauchy problem for
nondegenerate parabolic integro-differential equations in the scale of
generalized H\"{o}lder spaces}
\author{Remigijus Mikulevi\v{c}ius}
\email{mikulvcs@usc.edu}
\address{Department of Mathematics, University of Southern California, Los Angeles}

\author{Fanhui Xu}
\email{fanhuixu@usc.edu}
\address{Department of Mathematics, University of Southern California, Los Angeles}
\date{August 28, 2018}

\begin{abstract}
Parabolic integro-differential nondegenerate Cauchy problem is considered
in the scale of H\"{o}lder spaces of functions whose regularity is defined
by a radially O-regularly varying L\'{e}vy measure. Existence and uniqueness
and the estimates of the solution are derived.
\end{abstract}

\maketitle
\tableofcontents

\section{Introduction}

Let $\alpha \in \left( 0,2\right) $ and $\mathfrak{A}^{\alpha }$ be the
class of all nonnegative measures $\nu $\ on $\mathbf{R}_{0}^{d}=\mathbf{R}%
^{d}\backslash \left\{ 0\right\} $ such that $\int \left\vert y\right\vert
^{2}\wedge 1d\nu <\infty $ and 
\begin{equation*}
\alpha =\inf \left\{ \sigma <2:\int_{\left\vert y\right\vert \leq
1}\left\vert y\right\vert ^{\sigma }d\mathfrak{\nu }<\infty \right\} .
\end{equation*}%
In addition, we assume that for $\nu \in \mathfrak{A}^{\alpha },$ 
\begin{eqnarray*}
\int_{\left\vert y\right\vert >1}\left\vert y\right\vert d\nu  &<&\infty 
\text{ if }\alpha \in \left( 1,2\right) , \\
\int_{R<\left\vert y\right\vert \leq R^{\prime }}yd\nu  &=&0\text{ if }%
\alpha =1\text{ for all }0<R<R^{\prime }<\infty .\text{ }
\end{eqnarray*}

In this paper we consider the parabolic Cauchy problem with $\lambda \geq 0$%
\begin{eqnarray}
\partial _{t}u(t,x) &=&Lu(t,x)-\lambda u\left( t,x\right) +f(t,x)\text{ in }%
H_{T}=[0,T]\times \mathbf{R}^{d},  \label{1'} \\
u(0,x) &=&0,~x\in \mathbf{R}^{d},  \notag
\end{eqnarray}%
and integro-differential operator 
\begin{equation*}
L\varphi \left( x\right) =L^{\nu }\varphi \left( x\right) =\int \left[
\varphi (x+y)-\varphi \left( x\right) -\chi _{\alpha }\left( y\right) y\cdot
\nabla \varphi \left( x\right) \right] \nu \left( dy\right) ,\varphi \in
C_{0}^{\infty }\left( \mathbf{R}^{d}\right) ,
\end{equation*}%
where $\nu \in \mathfrak{A}^{\alpha },$ $\chi _{\alpha }\left( y\right) =0$
if $\alpha \in (0,1),\chi _{\alpha }\left( y\right) =1_{\left\{ \left\vert
y\right\vert \leq 1\right\} }\left( y\right) $ if $\alpha =1,$ and $\chi
_{\alpha }\left( y\right) =1$ if $\alpha \in (1,2).$ Given a L\'{e}vy
measure $\nu \in \mathfrak{A}^{\alpha }$ on $\mathbf{R}_{0}^{d}=\mathbf{R}%
^{d}\backslash \{0\}$, there exist a Poisson random measure $J\left(
ds,dy\right) $ on $[0,\infty )\times \mathbf{R}_{0}^{d}$ such that 
\begin{equation*}
\qquad \mathbf{E}\left[ J\left( ds,dy\right) \right] =\nu \left( dy\right)
ds,
\end{equation*}%
and a L\'{e}vy process $Z_{t}^{\nu }$ so that 
\begin{equation}
Z_{t}^{\nu }=\int_{0}^{t}\int_{\mathbf{R}_{0}^{d}}\chi _{\alpha }\left(
y\right) y\tilde{J}\left( ds,dy\right) +\int_{0}^{t}\int_{\mathbf{R}%
_{0}^{d}}\left( 1-\chi _{\alpha }\left( y\right) \right) yJ\left(
ds,dy\right) ,~t\geq 0,  \label{2}
\end{equation}%
with $\tilde{J}\left( ds,dy\right) =J\left( ds,dy\right) -\nu \left(
dy\right) ds.$ For $\nu \in \mathfrak{A}^{\alpha }$, set 
\begin{eqnarray*}
\delta \left( r\right)  &=&\delta _{\nu }\left( r\right) =\nu \left( \left\{
\left\vert y\right\vert >r\right\} \right) ,r>0, \\
w\left( r\right)  &=&w_{\nu }\left( r\right) =\delta _{\nu }\left( r\right)
^{-1},r>0.
\end{eqnarray*}%
Our main assumption is that $w\left( r\right) =w_{\nu }\left( r\right)
=\delta _{\nu }\left( r\right) ^{-1},r>0,$ is an O-RV function (O-regular
variation function) at zero (see \cite{aa} and \cite{bgt}), that is%
\begin{equation*}
r_{1}\left( \varepsilon \right) =\overline{\lim_{x\rightarrow 0}}\frac{%
\delta \left( \varepsilon x\right) ^{-1}}{\delta \left( x\right) ^{-1}}%
<\infty ,\varepsilon >0.
\end{equation*}%
By Theorem 2 in \cite{aa}, the following limits exist:%
\begin{equation}
p_{1}=p_{1}^{\nu }=\lim_{\varepsilon \rightarrow 0}\frac{\log r_{1}\left(
\varepsilon \right) }{\log \varepsilon },q_{1}=q_{1}^{\nu
}=\lim_{\varepsilon \rightarrow \infty }\frac{\log r_{1}\left( \varepsilon
\right) }{\log \varepsilon },  \label{if1}
\end{equation}%
and $\,p_{1}\leq q_{1}.$ It can be shown (see Remark \ref{ar1}) that $%
p_{1}\leq \alpha \leq q_{1}$. In this paper, we study the Cauchy problem (%
\ref{1'}) in the scale of spaces of generalized H\"{o}lder functions whose
regularity is determined by the L\'{e}vy measure $\nu$. We use $w$ to define
generalized Besov norms $\left\vert \cdot \right\vert _{\beta ,\infty }$ and
generalized spaces $\tilde{C}_{\infty ,\infty }^{\beta }\left( H_{T}\right)
,\beta >0$ (See Section 2.2.). They are Besov spaces of generalized
smoothness (see \cite{kg}, \cite{kg2}, \cite{fw2}) with admissible sequence $%
w\left( N^{-j}\right) ^{-\beta },j\geq 0,$ and covering sequence $%
N^{j},j\geq 0,$ with $N>1$. In particular (see Section 2), for $\beta \in
\left( 0,q_{1}^{-1}\right) $, the norm $\left\vert \cdot \right\vert _{\beta
,\infty }$ for the functions on $\mathbf{R}^{d}$ is equivalent to%
\begin{equation*}
\left\vert \left\vert u\right\vert \right\vert _{\beta }=\sup_{x}\left\vert
u\left( x\right) \right\vert +\sup_{x\neq y}\frac{\left\vert u\left(
x\right) -u\left( y\right) \right\vert }{w\left( \left\vert x-y\right\vert
\right) ^{\beta }}.
\end{equation*}%
When $\nu $ is \textquotedblleft close\textquotedblright to an $\alpha $-stable measure, they reduce to the
classical Besov (or equiv. H\"{o}lder-Zygmund) spaces.

Let

\begin{equation*}
\tilde{\nu}_{R}\left( dy\right) =w\left( R\right) \nu \left( Rdy\right)
,R\in \left[ 0,1\right] .
\end{equation*}
The main results of this paper is

\begin{theorem}
\label{thm2}Let $\beta \in \left( 0,\infty \right) ,\lambda \geq 0$. Let $%
\nu \in \mathfrak{A}^{\alpha },$ and \thinspace $w=w_{\nu }$ be an O-RV
function at zero with $p_{1},q_{1}$ defined in (\ref{if1}). Assume

\textbf{A.}%
\begin{eqnarray*}
0 &<&p_{1}\leq q_{1}<1\text{ if }\alpha \in \left( 0,1\right) ,~1\leq
p_{1}\leq q_{1}<2\text{ if }\alpha =1, \\
1 &<&p_{1}\leq q_{1}<2\text{ if }\alpha \in \left( 1,2\right) ;
\end{eqnarray*}

\textbf{B.} 
\begin{equation*}
\inf_{R\in (0,1],\left\vert \hat{\xi}\right\vert =1}\int_{\left\vert
y\right\vert \leq 1}\left\vert \hat{\xi}\cdot y\right\vert ^{2}\tilde{\nu}%
_{R}\left( dy\right) >0;
\end{equation*}

\textbf{C. }There is $N_{0}>2$ so that%
\begin{equation*}
\int_{1}^{\infty }w\left( t\right) ^{\frac{1}{q_{1}}}\frac{dt}{t^{N_{0}}}%
<\infty .
\end{equation*}

Then for each $f\in \tilde{C}_{\infty ,\infty }^{\beta }\left( H_{T}\right) $
there is a unique solution $u\in \tilde{C}_{\infty ,\infty }^{1+\beta
}\left( H_{T}\right) $ solving (\ref{1'}). Moreover, 
\begin{eqnarray}
\left\vert u\right\vert _{\beta ,\infty } &\leq &C\rho _{\lambda }\left(
T\right) \left\vert f\right\vert _{\beta ,\infty },  \label{est6} \\
\left\vert u\right\vert _{1+\beta ,\infty } &\leq &C\left[ 1+\rho _{\lambda
}\left( T\right) \right] ~\left\vert f\right\vert _{\beta ,\infty }
\label{est3}
\end{eqnarray}%
and 
\begin{eqnarray*}
&&\left\vert u\left( t,\cdot \right) -u\left( t^{\prime },\cdot \right)
\right\vert _{\mu +\beta ,\infty } \\
&\leq &C\left\{ \left( t-t^{\prime }\right) ^{1-\mu }+\left[ 1+\rho
_{\lambda }\left( T\right) \right] \left\vert t-t^{\prime }\right\vert
\right\} \left\vert f\right\vert _{\beta ,\infty }
\end{eqnarray*}%
for any $\mu \in \lbrack 0,1]$ and $t^{\prime }<t\leq T,$ where $\rho _{\lambda
}\left( T\right) =\left( \frac{1}{\lambda }\wedge T\right) $. The constant $C$
does not depend on $f,\lambda ,T,\mu $.
\end{theorem}

More specific examples could be the following.

\begin{example}
\label{e1}According to \cite{dm}, Chapter 3, 70-74, any L\'{e}vy measure $%
\nu \in \mathfrak{A}^{\alpha }$ can be disintegrated as%
\begin{equation*}
\nu \left( \Gamma \right) =-\int_{0}^{\infty }\int_{S_{d-1}}\chi _{\Gamma
}\left( rw\right) \Pi \left( r,dw\right) d\delta \left( r\right) ,\Gamma \in 
\mathcal{B}\left( \mathbf{R}_{0}^{d}\right) ,
\end{equation*}%
where $\delta =\delta _{\nu }$, and $\Pi \left( r,dw\right) ,r>0,$ is a
measurable family of measures on the unit sphere $S_{d-1}$ with $\Pi \left(
r,S_{d-1}\right) =1,r>0.$ If $\delta $ is an O-RV function, $\left\vert
\left\{ s\in \left[ 0,1\right] :r_{1}\left( s\right) <1\right\} \right\vert
>0,$ \textbf{A}, \textbf{C} and%
\begin{equation*}
\inf_{\left\vert \hat{\xi}\right\vert =1}\int_{S_{d-1}}\left\vert \hat{\xi}%
\cdot w\right\vert ^{2}\Pi \left( r,dw\right) \geq c_{0}>0,~r>0,
\end{equation*}%
hold, then all assumptions of Theorem \ref{thm2} are satisfied (see
Corollary \ref{ac2}).
\end{example}

\begin{example}
\label{e2}Consider L\'{e}vy measures in radial and angular coordinate system (%
$r=\left\vert y\right\vert ,w=\frac{y}{\left\vert y\right\vert }$) in the
form 
\begin{equation*}
\nu \left( B\right) =\int_{0}^{\infty }\int_{\left\vert w\right\vert
=1}1_{B}\left( rw\right) a\left( r,w\right) j\left( r\right) r^{d-1}S\left(
dw\right) dr,B\in \mathcal{B}\left( \mathbf{R}_{0}^{d}\right) ,
\end{equation*}%
where $S\left( dw\right) $ is a finite measure on the unit sphere.

Assume

(i) There is $C>1,c>0,0<\delta _{1}\leq \delta _{2}<1,$ such that 
\begin{equation*}
C^{-1}\phi \left( r^{-2}\right) \leq j\left( r\right) r^{d}\leq C\phi \left(
r^{-2}\right)
\end{equation*}%
and for all $1<r\leq R,$ 
\begin{equation*}
c^{-1}\left( \frac{R}{r}\right) ^{\delta _{1}}\leq \frac{\phi \left(
R\right) }{\phi \left( r\right) }\leq c\left( \frac{R}{r}\right) ^{\delta
_{2}}.
\end{equation*}

\textbf{(}ii\textbf{) }There is a function $\rho _{0}\left( w\right) $
defined on the unit sphere such that $\rho _{0}\left( w\right) \leq a\left(
r,w\right) \leq 1,\forall r>0$, and for all $\left\vert \hat{\xi}\right\vert
=1$, 
\begin{equation*}
\int_{S^{d-1}}\left\vert \hat{\xi}\cdot w\right\vert ^{2}\rho _{0}\left(
w\right) S\left( dw\right) \geq c>0.
\end{equation*}%
Under these assumptions, it can be shown that \textbf{B }and\textbf{\ C}
hold, and $\delta _{v}$ is an O-RV function with $2\delta _{1}\leq p_{1}\leq
q_{1}\leq 2\delta _{2}.$ Among the options for $\phi $ could be (see \cite%
{ksv})\newline
(1) $\phi \left( r\right) =\Sigma _{i=1}^{n}r^{\alpha _{i}},\alpha _{i}\in
\left( 0,1\right) ,i=1,\ldots ,n$;\newline
(2) $\phi \left( r\right) =\left( r+r^{\alpha }\right) ^{\beta },\alpha
,\beta \in \left( 0,1\right) $;\newline
(3) $\phi \left( r\right) =r^{\alpha }\left( \ln \left( 1+r\right) \right)
^{\beta },\alpha \in \left( 0,1\right) ,\beta \in \left( 0,1-\alpha \right) $%
;\newline
(4)$\ \phi \left( r\right) =\left( r+m^{1/\alpha }\right) ^{\alpha
}-m,\alpha \in (0,1),m>0;$\newline
(5) $\phi \left( r\right) =\left[ \ln \left( \cosh \sqrt{r}\right) \right]
^{\alpha },\alpha \in \left( 0,1\right) $.
\end{example}

Equations in classical H\"{o}lder spaces with non-local nondegenerate
operators of the form%
\begin{eqnarray*}
\mathcal{L}u\left( x\right)  &=&1_{\alpha \in \left( 0,2\right) }\int \left[
u\left( x+y\right) -u\left( x\right) -1_{\alpha \geq 1}1_{\left\vert
y\right\vert \leq 1}y\cdot \nabla u\left( x\right) \right] m\left(
x,y\right) \nu \left( dy\right)  \\
&+&1_{\alpha =2}a^{ij}\left( x\right) \partial _{ij}^{2}u\left( x\right)
+1_{\alpha \geq 1}\tilde{b}^{i}\left( x\right) \partial _{i}u\left( x\right)
+l\left( x\right) u\left( x\right) ,x\in \mathbf{R}^{d},
\end{eqnarray*}%
were considered in many papers. In \cite{ak}, the existence and uniqueness
of a solution to a parabolic equation with $\mathcal{L}$ in H\"{o}lder
spaces was proved analytically for $m$ H\"{o}lder continuous in $x$ and
smooth in $y$, $\nu \left( dy\right) =dy/\left\vert y\right\vert ^{d+\alpha
}.$ The elliptic problem $\mathcal{L}u=f$ in $\mathbf{R}^{d}$ with $\nu
\left( dy\right) =dy/\left\vert y\right\vert ^{d+\alpha }$ was considered in 
\cite{ba}, \cite{cas} and \cite{dk}. In \cite{cas}, the interior H\"{o}lder
estimates (in a non-linear case as well) were studied assuming that $m$ is
symmetric in $y$. In \cite{ba}, with $\nu \left( dy\right) =dy/\left\vert
y\right\vert ^{d+\alpha },$ the a priori estimates were derived in H\"{o}%
lder classes assuming H\"{o}lder continuity of $m$ in $x$, except the case $%
\alpha =1$. Similar results, including the case $\alpha =1$ were proved in 
\cite{dk}. In \cite{zh} (see references therein), in the classical H\"{o}%
lder spaces the case of a nondegenerate 
\begin{equation*}
\nu \left( \Gamma \right) =\int_{0}^{\infty }\int_{S_{d-1}}\chi _{\Gamma
}\left( rw\right) a\left( r,w\right) S\left( dw\right) \frac{dr}{r^{1+\alpha
}},\Gamma \in \mathcal{B}\left( \mathbf{R}_{0}^{d}\right) ,
\end{equation*}%
with a finite measure $S\left( dw\right) $ on the unit sphere was
considered. Finally, in \cite{rf}, for (\ref{1'}) with $x$-dependent density 
$m\left( x,y\right) $ at $\nu ,$ under different assumptions than \textbf{A}-%
\textbf{C}, existence and uniqueness in generalized smoothness classes is
derived.

Our paper is organized as follows. In section 2, notation is introduced, the
scale of generalized H\"{o}lder function spaces is defined, and various
equivalent norms are introduced. In particular, using some probabilistic
considerations, we prove the equivalence of $\left\vert u\right\vert _{\beta
,\infty }$ to the norms involving fractional powers of nondegenerate $L^{\nu
}$. The continuity of the operator is proved as well. Study of function
spaces of generalized smoothness dates back to the seventies-eighties, (see 
\cite{kg}, \cite{kg2} and references therein). Later, this interest
continued in connection with the construction problems of Markov processes
with jumps (see \cite{fw2}, \cite{fw} and references therein). In section 3,
we prove the main theorem by starting with smooth input functions. Then we
derive the key uniform estimates for the corresponding smooth solutions to (%
\ref{1'}). We handle generalized H\"{o}lder inputs by passing to the limit.
Finally, Appendix contains all needed results about O-RV functions. The
regular variation functions were introduced in \cite{k} and used in
tauberian theorems which were extended to O-RV functions as well (see \cite%
{aa}, \cite{bgt}, and references therein). They are very convenient for the
derivation of our main estimates.

\section{Notation, function spaces and norm equivalence}

\subsection{Basic notation}

We denote $\mathbf{N}=\{0,1,2,3,\ldots ,\}$, $\mathbf{N}_{+}=\mathbf{N}%
\backslash \{0\}$; $H_{T}=\left[ 0,T\right] \times \mathbf{R}^{d}$; $S^{d-1}$
is the unit sphere in $\mathbf{R}^{d}$. For $B\in \mathcal{B}\left( \mathbf{R%
}^{d}\right) $, we denote by $\left\vert B\right\vert $ the Lebesgue measure
of $B.$

For a function $u$ on $H_{T}$, we denote its partial derivatives by $%
\partial _{t}u=\partial u/\partial t$, $\partial _{i}u=\partial u/\partial
x_{i}$, $\partial _{ij}^{2}u=\partial ^{2}u/\partial x_{i}x_{j}$, and denote
its gradient with respect to $x$ by $\nabla u=\left( \partial _{1}u,\ldots
,\partial _{d}u\right) $ and $D^{|\gamma |}u=\partial ^{|\gamma |}u/\partial
x_{1}^{\gamma _{1}}\ldots \partial x_{d}^{\gamma _{d}}$, where $\gamma
=\left( \gamma _{1},\ldots ,\gamma _{d}\right) \in \mathbf{N}^{d}$ is a
multi-index.

We use $C_{b}^{\infty }\left( \mathbf{R}^{d}\right) $ to denote the set of
bounded infinitely differentiable functions on $\mathbf{R}^{d}$ whose
derivative of arbitrary order is bounded, and $C^{k}\left( \mathbf{R}%
^{d}\right) ,k\in \mathbf{N}$ the class of $k$-times continuously
differentiable functions.

We denote $\mathcal{S}\left( \mathbf{R}^{d}\right) $ the Schwartz space of
rapidly decreasing functions on $\mathbf{R}^{d},$ and $\mathcal{S}^{\prime
}\left( \mathbf{R}^{d}\right) $ denotes the space of continuous functionals
on $\mathcal{S}\left( \mathbf{R}^{d}\right) $, i.e. the space of tempered
distributions.

We adopt the normalized definition for Fourier and its inverse transforms
for functions in $\mathcal{S}\left( \mathbf{R}^d\right)$, i.e., 
\begin{eqnarray*}
\mathcal{F}\varphi\left(\xi\right) &=& \hat{\varphi}\left(\xi\right) :=\int
e^{-i2\pi x\cdot \xi}\varphi\left(x\right)dx, \\
\mathcal{F}^{-1}\varphi\left(x\right)&=&\check{\varphi}\left(x\right) :=
\int e^{i2\pi x\cdot \xi}\varphi\left(\xi\right)d\xi, \enskip \varphi\in%
\mathcal{S}\left( \mathbf{R}^d\right).
\end{eqnarray*}
Recall that Fourier transform can be extended to a bijection on $\mathcal{S}%
^{\prime }\left( \mathbf{R}^d\right)$.

Throughout the sequel, $Z_{t}^{\nu }$ represents the L\'{e}vy process
associated to the L\'{e}vy measure $\nu \in \mathfrak{A}^{\alpha },$ see (%
\ref{2}).

For any L\'{e}vy measure $\nu \in \mathfrak{A}^{\alpha }$ and $R>0$, 
\begin{equation}
  \nu _{R}\left( B\right) :=\int 1_{B}\left( y/R\right) \nu \left( dy\right)
  ,B\in \mathcal{B}\left( \mathbf{R}_{0}^{d}\right) ,~\tilde{\nu}_{R}\left(
  dy\right):=w\left( R\right) \nu _{R}\left( dy\right) .\label{meas}
\end{equation}

For any L\'{e}vy measure $\nu \in \mathfrak{A}^{\alpha }$, we denote its symmetrization 
\begin{equation*}
\nu _{sym}\left( dy\right) =\frac{1}{2}\left[ \nu \left( dy\right) +\nu
\left( -dy\right) \right].
\end{equation*}%
And $\mathfrak{A}_{sym}^{\alpha }=\left\{ \mu \in \mathfrak{A}^{\alpha }:\mu
=\mu _{sym}\right\} .$

If $A\left( \mathbf{R}^{d}\right) $ is a space of functions $v$ on $\mathbf{R%
}^{d}$ with norm $\left\vert v\right\vert _{A}=\left\vert v\right\vert
_{A\left( \mathbf{R}^{d}\right) },$ then $A\left( H_{T}\right) $ denotes the
spaces of functions $u$ on $H_{T}=\left[ 0,T\right] \times \mathbf{R}^{d}$
with finite norm%
\begin{equation*}
\left\vert u\right\vert _{A}=\left\vert u\right\vert _{A\left( H_{T}\right)
}=\sup_{t\in \left[ 0,T\right] }\left\vert u\left( t,\cdot \right)
\right\vert _{A\left( \mathbf{R}^{d}\right) }.
\end{equation*}

We have specific values assigned for $c_{0},c_{1},c_{2},N_{0},N_{1}$, but we
allow $C$ to vary from line to line. In particular, $C\left( \cdots \right) $
represents a constant depending only on quantities in the parentheses.

\subsection{Function spaces and norm equivalence}

We fix a constant $N>1$. For such an $N$, by Lemma 6.1.7 in \cite{bl} and
appropriate scaling, there exists $\phi \in C_{0}^{\infty }\left( \mathbf{R}%
^{d}\right) $ such that supp$\left( \phi \right) =\{\xi :\frac{1}{N}\leq
\left\vert \xi \right\vert \leq N\}$, $\phi \left( \xi \right) >0$ in the
interior of its support, and 
\begin{equation*}
\sum_{j=-\infty }^{\infty }\phi \left( N^{-j}\xi \right) =1\mbox{ if }\xi
\neq 0.
\end{equation*}%
We denote throughout this paper 
\begin{eqnarray}
&&\varphi _{j}=\mathcal{F}^{-1}\left[ \phi \left( N^{-j}\xi \right) \right]
,\quad j=1,2,\ldots ,\xi \in \mathbf{R}^{d},  \label{j1} \\
&&\varphi _{0}=\mathcal{F}^{-1}\left[ 1-\sum_{j=1}^{\infty }\phi \left(
N^{-j}\xi \right) \right] .  \label{j0}
\end{eqnarray}%
Apparently, $\varphi _{j}\in \mathcal{S}\left( \mathbf{R}^{d}\right) ,j\in 
\mathbf{N}$. They are convolution functions we use to define \emph{%
generalized Besov spaces}. Namely, for $\beta >0$ we write $\tilde{C}%
_{\infty ,\infty }^{\beta }\left( \mathbf{R}^{d}\right) $ as the set of
functions in $\mathcal{S}^{\prime }\left( \mathbf{R}^{d}\right) $ for which 
\begin{equation}
\left\vert u\right\vert _{\beta ,\infty }:=\sup_{j}w\left( N^{-j}\right)
^{-\beta }\left\vert u\ast \varphi _{j}\right\vert _{0}<\infty ,  \label{j2}
\end{equation}%
where $w=w_{\nu }$ with $\nu \in \mathfrak{A}^{\alpha }.$

\begin{lemma}
\label{equiv}Let $\nu \in \mathfrak{A}^{\alpha }$, $w=w_{\nu }$ be an O-RV
function at zero and \textbf{A} holds for it. Let $\beta \in \left( 0,\infty
\right) $. If $u\in \tilde{C}_{\infty ,\infty }^{\beta }\left( \mathbf{R}%
^{d}\right) $, then $u$ is bounded and continuous, 
\begin{equation*}
u\left( x\right) =\sum_{j=0}^{\infty }\left( u\ast \varphi _{j}\right)
\left( x\right) ,x\in \mathbf{R}^{d},
\end{equation*}%
where the series converges uniformly. Moreover,%
\begin{equation*}
\left\vert u\right\vert _{0}\leq C\left\vert u\right\vert _{\beta ,\infty
},u\in \tilde{C}_{\infty ,\infty }^{\beta }\left( \mathbf{R}^{d}\right) .
\end{equation*}
\end{lemma}

\begin{proof}
Note that $u\ast \varphi _{j}$ is continuous of moderate growth and $%
\sum_{j=0}^{\infty }u\ast \varphi _{j}=u$ in $\mathcal{S}^{\prime }\left( 
\mathbf{R}^{d}\right) $. Obviously, by Corollary \ref{ac1} in Appendix, 
\begin{eqnarray*}
\sum_{j=0}^{\infty }\left\vert u\ast \varphi _{j}\right\vert _{0}
&=&\sum_{j=0}^{\infty }w\left( N^{-j}\right) ^{\beta }w\left( N^{-j}\right)
^{-\beta }\left\vert u\ast \varphi _{j}\right\vert _{0} \\
&\leq &\sup_{j\geq 0}w\left( N^{-j}\right) ^{-\beta }\left\vert u\ast
\varphi _{j}\right\vert _{0}\sum_{j=0}^{\infty }w\left( N^{-j}\right)
^{\beta } \\
&\leq &C\left\vert u\right\vert _{\beta ,\infty }\sum_{j=0}^{\infty }w\left(
N^{-j}\right) ^{\beta }\leq C\left\vert u\right\vert _{\beta ,\infty }.
\end{eqnarray*}
\end{proof}

Let $\nu \in \mathfrak{A}^{\alpha }$, $w=w_{\nu },\beta >0$. For $u\in
C_{b}^{\infty }\left( \mathbf{R}^{d}\right) $, set 
\begin{equation*}
\left\vert u\right\vert _{0}=\sup_{x}\left\vert u\left( x\right) \right\vert
,~\left[ u\right] _{\beta }=\sup_{x,h\neq 0}\frac{\left\vert u\left(
x+h\right) -u\left( x\right) \right\vert }{w\left( \left\vert h\right\vert
\right) ^{\beta }},
\end{equation*}%
and%
\begin{equation*}
\left\vert u\right\vert _{\beta }:=\left\vert u\right\vert _{0}+\left[ u%
\right] _{\beta }.
\end{equation*}

\begin{proposition}
\label{pr2}Let $\nu \in \mathfrak{A}^{\alpha }$, $w=w_{\nu }$ be an O-RV
function at zero so that \textbf{A} and\textbf{\ C} hold for it. Let $\beta
\in \left( 0,q_{1}^{-1}\right) .$ Then the norm $\left\vert u\right\vert
_{\beta }$ and norm $\left\vert u\right\vert _{\beta ,\infty }$ are
equivalent on $C_{b}^{\infty }\left( \mathbf{R}^{d}\right) $. Namely, there
is $C>0$ depending only on $d,\beta ,N$ such that 
\begin{equation*}
C^{-1}\left\vert u\right\vert _{\beta }\leq \left\vert u\right\vert _{\beta
,\infty }\leq C\left\vert u\right\vert _{\beta },u\in C_{b}^{\infty }\left( 
\mathbf{R}^{d}\right) .
\end{equation*}
\end{proposition}

\begin{proof}
Let $u\in C_{b}^{\infty }\left( \mathbf{R}^{d}\right) $. Then, by Lemma \ref%
{al1}, $\left\vert u\right\vert _{\beta }<\infty $. If $j=0$, then 
\begin{equation*}
\left\vert u\ast \varphi _{0}\right\vert _{0}\leq \left\vert u\right\vert
_{0}\int \left\vert \varphi _{0}\left( y\right) \right\vert dy\leq
C\left\vert u\right\vert _{\beta }.
\end{equation*}%
If $j\neq 0$, then by the construction of $\varphi _{j}$, $\int \varphi
_{j}\left( y\right) dy=\hat{\varphi}_{j}\left( 0\right) =0$. Therefore,
denoting $\varphi =\mathcal{F}^{-1}\phi $, 
\begin{eqnarray*}
&&\left\vert u\ast \varphi _{j}\right\vert _{0} \\
&=&\left\vert \int \left[ u\left( y\right) -u\left( x\right) \right] \varphi
_{j}\left( x-y\right) dy\right\vert _{0} \\
&\leq &\left[ u\right] _{\beta }\int w\left( \left\vert y-x\right\vert
\right) ^{\beta }N^{jd}\left\vert \varphi \left( N^{j}\left( x-y\right)
\right) \right\vert dy \\
&=&\left[ u\right] _{\beta }\int w\left( N^{-j}\left\vert y\right\vert
\right) ^{\beta }\left\vert \varphi \left( y\right) \right\vert dy.
\end{eqnarray*}%
Since for $N_{0}>d+1,$%
\begin{equation*}
\left\vert \varphi \left( y\right) \right\vert \leq C\left( 1+\left\vert
y\right\vert \right) ^{-N_{0}},y\in \mathbf{R}^{d}\mathbf{,}
\end{equation*}%
for some $C>0,$ we have 
\begin{eqnarray*}
&&\int w\left( N^{-j}\left\vert y\right\vert \right) ^{\beta }\left\vert
\varphi \left( y\right) \right\vert dy \\
&\leq &C\int_{0}^{1}w\left( N^{-j}\left\vert y\right\vert \right) ^{\beta
}dy+C\int_{1}^{\infty }w\left( N^{-j}\left\vert y\right\vert \right) ^{\beta
}\left\vert y\right\vert ^{-N_{0}}dy=A_{1}+A_{2}.
\end{eqnarray*}

By Lemma \ref{al1} 
\begin{equation*}
N^{-j(N_{0}-d)}\leq Cw\left( N^{-j}\right) ^{\beta },j\geq 0,
\end{equation*}%
and, 
\begin{eqnarray*}
A_{1} &\leq &CN^{jd}\int_{0}^{N^{-j}}w\left( r\right) ^{\beta }r^{d}\frac{dr%
}{r}\leq Cw\left( N^{-j}\right) ^{\beta }, \\
A_{2} &=&CN^{-j(N_{0}-d)}\int_{N^{-j}}^{\infty }w\left( r\right) ^{\beta
}r^{-(N_{0}-d)}\frac{dr}{r} \\
&=&CN^{-j(N_{0}-d)}\int_{N^{-j}}^{1}w\left( r\right) ^{\beta }r^{-(N_{0}-d)}%
\frac{dr}{r}+N^{-j(N_{0}-d)}\int_{1}^{\infty }w\left( r\right) ^{\beta
}r^{-(N_{0}-d)}\frac{dr}{r} \\
&\leq &Cw\left( N^{-j}\right) ^{\beta },j\geq 0.
\end{eqnarray*}%
That is to say $\left\vert u\right\vert _{\beta ,\infty }\leq C\left\vert
u\right\vert _{\beta },u\in C_{b}\left( \mathbf{R}^{d}\right) $ for some
constant $C\left( \beta ,d\right) >0$.

Let $\tilde{\phi},\tilde{\phi}_{0}\in C_{0}^{\infty }\left( \mathbf{R}%
^{d}\right) $, be such that $0\notin $supp$\left( \tilde{\phi}\right) ,$ $%
\tilde{\phi}\phi =\phi ,\tilde{\phi}_{0}\phi _{0}=\phi _{0}$, where $\phi
_{0}=\mathcal{F}\varphi _{0},$ and $\phi ,\varphi _{0}$ are the functions
introduced in (\ref{j0}), (\ref{j1}). Let 
\begin{eqnarray}
\tilde{\varphi}=\mathcal{F}^{-1}\tilde{\phi}, &&\tilde{\varphi _{j}}=%
\mathcal{F}^{-1}\tilde{\phi}\left( N^{-j}\cdot \right) ,j\geq 1,
\label{ssch} \\
&&\tilde{\varphi}_{0}=\mathcal{F}^{-1}\tilde{\phi}_{0}.  \label{sch}
\end{eqnarray}%
Hence 
\begin{equation*}
\varphi _{j}=\varphi _{j}\ast \tilde{\varphi}_{j},j\geq 0,
\end{equation*}%
where in particular, 
\begin{equation*}
\tilde{\varphi}_{j}\left( x\right) =N^{jd}\tilde{\varphi}\left(
N^{j}x\right) ,j\geq 1,x\in \mathbf{R}^{d}.
\end{equation*}

Obviously, 
\begin{eqnarray*}
&&\left\vert u\ast \varphi _{0}\left( x+y\right) -u\ast \varphi _{0}\left(
x\right) \right\vert  \\
&\leq &\int \left\vert \tilde{\varphi}_{0}\left( x+y-z\right) -\tilde{\varphi%
}_{0}\left( x-z\right) \right\vert \left\vert u\ast \varphi _{0}\left(
z\right) \right\vert dz \\
&\leq &C\left( \left\vert y\right\vert \wedge 1\right) \left\vert u\ast
\varphi _{0}\right\vert _{0},~x,y\in \mathbf{R}^{d},
\end{eqnarray*}%
and 
\begin{eqnarray*}
&&\left\vert u\ast \varphi _{j}\left( x+y\right) -u\ast \varphi _{j}\left(
x\right) \right\vert  \\
&\leq &N^{jd}\int \left\vert \tilde{\varphi}\left( N^{j}\left( x+y-z\right)
\right) -\tilde{\varphi}\left( N^{j}\left( x-z\right) \right) \right\vert
\left\vert u\ast \varphi _{j}\left( z\right) \right\vert dz \\
&\leq &C\left( \left\vert N^{j}y\right\vert \wedge 1\right) \left\vert u\ast
\varphi _{j}\right\vert _{0},j\geq 1,~x,y\in \mathbf{R}^{d}.
\end{eqnarray*}%
By Lemma \ref{equiv}, for $x,y\in \mathbf{R}^{d},$ 
\begin{eqnarray*}
&&\left\vert u\left( x+y\right) -u\left( x\right) \right\vert  \\
&\leq &\sum_{j=0}^{\infty }\left\vert u\ast \varphi _{j}\left( x+y\right)
-u\ast \varphi _{j}\left( x\right) \right\vert \leq C\sum_{j=0}^{\infty
}\left( N^{j}\left\vert y\right\vert \wedge 1\right) \left\vert u\ast
\varphi _{j}\right\vert _{0}.
\end{eqnarray*}%
Let $\beta q_{1}<1,k\in \mathbf{N}$. For $\left\vert y\right\vert \in
(N^{-k-1},N^{-k}],$ 
\begin{eqnarray*}
\left\vert u\left( x+y\right) -u\left( x\right) \right\vert  &\leq
&C\left\vert u\right\vert _{\beta ,\infty }\sup_{\left\vert y\right\vert
\leq N^{-k}}\sum_{j=0}^{\infty }\left( N^{j}\left\vert y\right\vert \wedge
1\right) w\left( N^{-j}\right) ^{\beta } \\
&\leq &C\left\vert u\right\vert _{\beta ,\infty }\left[
\sum_{j=0}^{k}N^{j-k}w\left( N^{-j}\right) ^{\beta }+\sum_{j=k+1}^{\infty
}w\left( N^{-j}\right) ^{\beta }\right] .
\end{eqnarray*}

Then, by Lemma \ref{al1}, 
\begin{eqnarray*}
N^{-k}\sum_{j=0}^{k}N^{j}w\left( N^{-j}\right) ^{\beta } &\leq
&C_{2}N^{-k}\int_{0}^{k+1}N^{x}w\left( N^{-x}\right) ^{\beta }dx \\
&\leq &CN^{-k}\int_{N^{-k-1}}^{1}x^{-1}w\left( x\right) ^{\beta }\frac{dx}{x}%
\leq Cw\left( \left\vert y\right\vert \right) ^{\beta }.
\end{eqnarray*}%
Again, by Lemma \ref{al1}, 
\begin{equation*}
\sum_{j=k+1}^{\infty }w\left( N^{-j}\right) ^{\beta }\leq
C\int_{k+1}^{\infty }w\left( N^{-x}\right) ^{\beta }dx\leq
C\int_{0}^{N^{-k-1}}w\left( x\right) ^{\beta }\frac{dx}{x}\leq Cw\left(
\left\vert y\right\vert \right) ^{\beta }.
\end{equation*}%
The statement is proved.\noindent
\end{proof}

\subsubsection{Equivalent norms on $C_{b}^{\infty }\left( \mathbf{R}%
^{d}\right) $}

Now we will introduce some other norms on $C_{b}^{\infty }\left( \mathbf{R}%
^{d}\right) $ involving the powers of the operators $L^{\nu },I-L^{\nu }$: 
\begin{eqnarray*}
\left\vert u\right\vert _{\nu ,\kappa ,\beta } &=&\left\vert u\right\vert
_{\kappa ,\beta }=\left\vert u\right\vert _{0}+\left\vert L^{\nu ;\kappa
}u\right\vert _{\beta ,\infty },u\in C_{b}^{\infty }\left( \mathbf{R}%
^{d}\right) , \\
\left\vert \left\vert u\right\vert \right\vert _{\nu ;\kappa ,\beta }
&=&\left\vert \left\vert u\right\vert \right\vert _{\kappa ,\beta
}=\left\vert \left( I-L^{\nu }\right) ^{\kappa }u\right\vert _{\beta ,\infty
},u\in C_{b}^{\infty }\left( \mathbf{R}^{d}\right) ,
\end{eqnarray*}%
with $\kappa ,\beta >0$, $L^{\nu ;\kappa }=\left( L^{\nu }\right) ^{\kappa }$%
, and $\nu $ satisfying \textbf{A} and \textbf{B}. In addition, we assume
that $\nu \in \mathfrak{A}_{sym}^{\alpha }=\left\{ \mu \in \mathfrak{A}%
^{\alpha }:\mu =\mu _{sym}\right\} $ if $\kappa $ is fractional. First, we
define those powers and corresponding norms on $C_{b}^{\infty }\left( 
\mathbf{R}^{d}\right) $. Then we study their relations and extend them to $%
\tilde{C}_{\infty \infty }^{\kappa +\beta }\left( R^{d}\right) .$

For $\nu \in \mathfrak{A}_{sym}^{\alpha },\kappa \in (0,1),a\geq 0,$ and $%
f\in \mathcal{S}\left( \mathbf{R}^{d}\right) $, we see easily that 
\begin{eqnarray*}
&&\left( a-\psi ^{\nu }\left( \xi \right) \right) ^{\kappa }\hat{f}\left(
\xi \right)  \\
&=&c_{\kappa }\int_{0}^{\infty }t^{-\kappa }\left[ e^{-at}\exp \left( \psi
^{\nu }\left( \xi \right) t\right) -1\right] \frac{dt}{t}\hat{f}\left( \xi
\right) ,\xi \in \mathbf{R}^{d},
\end{eqnarray*}%
and define%
\begin{eqnarray}
&&\left( aI-L^{\nu }\right) ^{\kappa }f\left( x\right)   \label{ff01} \\
&=&\mathcal{F}^{-1}\left[ \left( a-\psi ^{\nu }\right) ^{\kappa }\hat{f}%
\right] \left( x\right)   \notag \\
&=&c_{\kappa }\mathbf{E}\int_{0}^{\infty }t^{-\kappa }\left[ e^{-at}f\left(
x+Z_{t}^{\nu }\right) -f\left( x\right) \right] \frac{dt}{t},x\in \mathbf{R}%
^{d},  \notag
\end{eqnarray}%
where%
\begin{equation*}
c_{\kappa }=\left( \int_{0}^{\infty }\left( e^{-t}-1\right) t^{-\kappa }%
\frac{dt}{t}\right) ^{-1}.
\end{equation*}%
We denote, with $a=0,f\in \mathcal{S}\left( \mathbf{R}^{d}\right) ,\kappa
\in \left( 0,1\right) ,$%
\begin{equation*}
L^{\nu ;\kappa }f:=\mathcal{F}^{-1}\left[ -\left( -\psi ^{\nu }\right)
^{\kappa }\hat{f}\right] .
\end{equation*}%
For $f\in C_{b}^{\infty }\left( \mathbf{R}^{d}\right) ,\kappa \in \left(
0,1\right) ,a\geq 0,$ we define%
\begin{equation*}
\left( aI-L^{\nu }\right) ^{\kappa }f\left( x\right) =c_{\kappa }\mathbf{E}%
\int_{0}^{\infty }t^{-\kappa }\left[ e^{-at}f\left( x+Z_{t}^{\nu }\right)
-f\left( x\right) \right] \frac{dt}{t},x\in \mathbf{R}^{d}.
\end{equation*}%
For $\kappa =1,$ $\left( aI-L^{\nu }\right) ^{1}f=\left( aI-L^{\nu }\right)
f=af-L^{\nu }f,f\in C_{b}^{\infty }\left( \mathbf{R}^{d}\right) .$ Note that
for $\kappa \in \left( 0,1\right) ,a\geq 0,$%
\begin{eqnarray}
&&\left( aI-L^{\nu }\right) ^{\kappa }f\left( x\right)   \label{ff2} \\
&=&c_{\kappa }\mathbf{E}\int_{1}^{\infty }t^{-\kappa }\left[ e^{-at}f\left(
x+Z_{t}^{\nu }\right) -f\left( x\right) \right] \frac{dt}{t}  \notag \\
&&+c_{\kappa }\mathbf{E}\int_{0}^{1}t^{-\kappa
}\int_{0}^{t}e^{-as}(-a+L^{\nu })f\left( x+Z_{s}^{\nu }\right) ds\frac{dt}{t}%
,%
\begin{array}{c}
x\in \mathbf{R}^{d}.%
\end{array}
\notag
\end{eqnarray}%
For $a>0,f\in C_{b}^{\infty }\left( \mathbf{R}^{d}\right) ,$ set 
\begin{equation*}
\left( aI-L^{\nu }\right) ^{-\kappa }f\left( x\right) =c_{\kappa }^{\prime
}\int_{0}^{\infty }t^{\kappa }e^{-at}\mathbf{E}f\left( x+Z_{t}^{\nu }\right) 
\frac{dt}{t},x\in \mathbf{R}^{d},
\end{equation*}%
where%
\begin{equation*}
c_{\kappa }^{\prime }=\left( \int_{0}^{\infty }t^{\kappa }e^{-t}\frac{dt}{t}%
\right) ^{-1},
\end{equation*}%
and $\nu \in \mathfrak{A}_{sym}^{\alpha },\kappa >0,$ or $\nu \in \mathfrak{A%
}^{\alpha },\kappa \in \mathbf{N}$.

Note that for $g\in \mathcal{S}\left( \mathbf{R}^{d}\right) ,$%
\begin{eqnarray*}
\mathcal{F}\left[ \left( aI-L^{\nu }\right) ^{-\kappa }g\right] &=&\left(
a-\psi ^{\nu }\right) ^{-\kappa }\hat{g},a>0,\kappa >0, \\
\mathcal{F}\left[ \left( aI-L^{\nu }\right) ^{\kappa }g\right] &=&\left(
a-\psi ^{\nu }\right) ^{\kappa }\hat{g},a\geq 0,\kappa \in (0,1].
\end{eqnarray*}

We use the formulas above to define $\left( a-L^{\nu }\right) ^{\kappa
},a\geq 0,\kappa =1,0,-1,\ldots ,$ for $\nu \in \mathfrak{A}^{\alpha }.$

\begin{remark}
\label{re0}Assume $\kappa \in (0,1],a\geq 0$ or $\kappa \in (-\infty ,0),a>0$%
. It is easy to see that

a) for any $f\in C_{b}^{\infty }\left( \mathbf{R}^{d}\right) $, we have $%
\left( aI-L^{\nu }\right) ^{\kappa }f\in C_{b}^{\infty }\left( \mathbf{R}%
^{d}\right) $ and for any multiindex $\gamma ,$ $D^{\gamma }\left( aI-L^{\nu
}\right) ^{\kappa }f=\left( aI-L^{\nu }\right) ^{\kappa }D^{\gamma }f$,$\nu
\in \mathfrak{A}_{sym}^{\alpha }.$ The same holds for $\nu \in \mathfrak{A}%
^{\alpha }$ and $\kappa =1,0,-1,\ldots .$

b) for any $f\in C_{b}^{\infty }\left( \mathbf{R}^{d}\right) $ such that for
any multiindex $\gamma $, $D^{\gamma }f\in L^{1}\left( \mathbf{R}^{d}\right)
\cap L^{2}\left( \mathbf{R}^{d}\right) $, we have 
\begin{eqnarray*}
\mathcal{F}\left[ \left( aI-L^{\nu }\right) ^{-\kappa }f\right]  &=&\left(
a-\psi ^{\nu }\right) ^{-\kappa }\hat{f},a>0,\kappa >0, \\
\mathcal{F}\left[ \left( aI-L^{\nu }\right) ^{\kappa }f\right]  &=&\left(
a-\psi ^{\nu }\right) ^{\kappa }\hat{f},a\geq 0,\kappa \in (0,1],
\end{eqnarray*}%
for $\nu \in \mathfrak{A}_{sym}^{\alpha }.$ The same holds for $\nu \in 
\mathfrak{A}^{\alpha }$ and $\kappa =1,0,-1,\ldots .$
\end{remark}

The following obvious claim holds.

\begin{lemma}
\label{le1}Let $\nu \in \mathfrak{A}_{sym}^{\alpha }$. Assume $\kappa \in
(0,1],a\geq 0$ or $\kappa \in (-\infty ,0),a>0$. Let $f,f_{n}\in
C_{b}^{\infty }\left( \mathbf{R}^{d}\right) $\ be so that for any multiindex 
$\gamma ,$ $D^{\gamma }f_{n}\rightarrow D^{\gamma }f$ as $n\rightarrow
\infty $ uniformly on compact subsets of $\mathbf{R}^{d}$ and 
\begin{equation*}
\sup_{x,n}\left\vert D^{\gamma }f_{n}\left( x\right) \right\vert <\infty .
\end{equation*}%
Then for any multiindex $\gamma ,$%
\begin{equation*}
D^{\gamma }\left( aI-L^{\nu }\right) ^{\kappa }f_{n}=\left( aI-L^{\nu
}\right) ^{\kappa }D^{\gamma }f_{n}\rightarrow D^{\gamma }\left( aI-L^{\nu
}\right) ^{\kappa }f=\left( aI-L^{\nu }\right) ^{\kappa }D^{\gamma }f
\end{equation*}%
uniformly on compact subsets of $\mathbf{R}^{d}$, and%
\begin{equation*}
\sup_{x,n}\left\vert \left( aI-L^{\nu }\right) ^{\kappa }D^{\gamma
}f_{n}\left( x\right) \right\vert <\infty .
\end{equation*}

The same holds for $\nu \in \mathfrak{A}^{\alpha }$ and $\kappa
=1,0,-1,\ldots .$
\end{lemma}

\begin{remark}
\label{re1}Given $f\in C_{b}^{\infty }\left( \mathbf{R}^{d}\right) $ there
is a sequence $f_{n}\in C_{0}^{\infty }\left( \mathbf{R}^{d}\right) $ so
that for any multiindex $\gamma ,$ $D^{\gamma }f_{n}\rightarrow D^{\gamma }f$
as $n\rightarrow \infty $ uniformly on compact subsets of $\mathbf{R}^{d}$
and 
\begin{equation*}
\sup_{x,n}\left\vert D^{\gamma }f_{n}\left( x\right) \right\vert <\infty .
\end{equation*}

Indeed, choose $g\in C_{b}^{\infty }\left( \mathbf{R}^{d}\right) ,0\leq
g\leq 1,g\left( x\right) =1$ if $\left\vert x\right\vert \leq 1$, and
\thinspace $g\left( x\right) =0$ if $\left\vert x\right\vert >2$. Given $%
f\in C_{b}^{\infty }\left( \mathbf{R}^{d}\right) ,$ take 
\begin{equation*}
f_{n}\left( x\right) =f\left( x\right) g\left( x/n\right) ,x\in \mathbf{R}%
^{d},n\geq 1.
\end{equation*}
\end{remark}

\begin{lemma}
\label{le2}Let $\nu \in \mathfrak{A}_{sym}^{\alpha }$. Assume $a>0,\kappa
\in (0,1]$. Then $\left( aI-L^{\nu }\right) ^{\kappa }:C_{b}^{\infty }\left( 
\mathbf{R}^{d}\right) \rightarrow C_{b}^{\infty }\left( \mathbf{R}%
^{d}\right) $ is bijective whose inverse is $\left( aI-L^{\nu }\right)
^{-\kappa }:$%
\begin{equation*}
\left( aI-L^{\nu }\right) ^{\kappa }\left( aI-L^{\nu }\right) ^{-\kappa
}f\left( x\right) =\left( aI-L^{\nu }\right) ^{-\kappa }\left( aI-L^{\nu
}\right) ^{\kappa }f\left( x\right) =f\left( x\right) ,x\in \mathbf{R}^{d},
\end{equation*}%
for any $f\in C_{b}^{\infty }\left( \mathbf{R}^{d}\right) .$
\end{lemma}

\begin{proof}
It is an easy consequence of Lemma \ref{le1} and Remarks \ref{re0} and \ref%
{re1}.
\end{proof}

For an integer $k\in \mathbf{N}$, we define for $\nu \in \mathfrak{A}%
^{\alpha },$%
\begin{equation*}
\left( aI-L^{\nu }\right) ^{k}=\underset{k\text{ times}}{\underbrace{\left(
aI-L^{\nu }\right) \ldots \left( aI-L^{\nu }\right) }}.
\end{equation*}

For a non integer $\kappa >0,\kappa =\left[ \kappa \right] +s$ with $s\in
\left( 0,1\right) $ and $\nu \in \mathfrak{A}_{sym}^{\alpha }$, we set 
\begin{eqnarray*}
\left( aI-L^{\nu }\right) ^{\kappa }f &=&\left( aI-L^{\nu }\right) ^{\left[
\kappa \right] }\left( aI-L^{\nu }\right) ^{s}f \\
&=&\left( aI-L^{\nu }\right) ^{s}\left( aI-L^{\nu }\right) ^{\left[ \kappa %
\right] }f,f\in C_{b}^{\infty }\left( \mathbf{R}^{d}\right) .
\end{eqnarray*}

\begin{remark}
\label{re2}Let $\nu \in \mathfrak{A}_{sym}^{\alpha }$, and $f\in
C_{b}^{\infty }\left( \mathbf{R}^{d}\right) $ be such that for any
multiindex $\gamma $, $D^{\gamma }f\in L^{1}\left( \mathbf{R}^{d}\right)
\cap L^{2}\left( \mathbf{R}^{d}\right) $. Then%
\begin{eqnarray*}
\mathcal{F}\left[ \left( aI-L^{\nu }\right) ^{-s}f\right]  &=&\left( a-\psi
^{\nu }\right) ^{-s}\hat{f},a>0,s>0, \\
\mathcal{F}\left[ \left( aI-L^{\nu }\right) ^{s}f\right]  &=&\left( a-\psi
^{\nu }\right) ^{s}\hat{f},a\geq 0,s>0.
\end{eqnarray*}%
The same holds with $\nu \in \mathfrak{A}^{\alpha }$ if $s\in \mathbf{N}$.
\end{remark}

\begin{lemma}
\label{le3}Assume $a>0$. Then

(i) for any $\kappa ,s\geq 0,$ we have $L^{\nu ;\kappa }L^{\nu ;s}=L^{\nu
;\kappa +s};$ for any $\kappa ,s\in \mathbf{R,}$%
\begin{eqnarray*}
\left( aI-L^{\nu }\right) ^{\kappa }\left( aI-L^{\nu }\right) ^{s} &=&\left(
aI-L^{\nu }\right) ^{\kappa +s}, \\
\left( aI-L^{\nu }\right) ^{-\kappa }\left( aI-L^{\nu }\right) ^{-s}
&=&\left( aI-L^{\nu }\right) ^{-(\kappa +s)},
\end{eqnarray*}%
for $\nu \in \mathfrak{A}_{sym}^{\alpha }$. 

The same holds with $\nu \in \mathfrak{A}^{\alpha }$ if $\kappa ,s\in 
\mathbf{N}$.

(ii) for any $\kappa >0$, the mapping $\left( aI-L^{\nu }\right) ^{\kappa
}:C_{b}^{\infty }\left( \mathbf{R}^{d}\right) \rightarrow C_{b}^{\infty
}\left( \mathbf{R}^{d}\right) $ is bijective whose inverse is $\left(
aI-L^{\nu }\right) ^{-\kappa }:$%
\begin{equation*}
\left( aI-L^{\nu }\right) ^{\kappa }\left( aI-L^{\nu }\right) ^{-\kappa
}f\left( x\right) =\left( aI-L^{\nu }\right) ^{-\kappa }\left( aI-L^{\nu
}\right) ^{\kappa }f\left( x\right) =f\left( x\right) ,x\in \mathbf{R}^{d}.
\end{equation*}%
for any $f\in C_{b}^{\infty }\left( \mathbf{R}^{d}\right) .$ 

The same holds with $\nu \in \mathfrak{A}^{\alpha }$ if $\kappa \in \mathbf{N%
}$.
\end{lemma}

\begin{proof}
The statement is an easy consequence of Lemma \ref{le1} and Remarks \ref{re0}%
, \ref{re1}, and \ref{re2}.
\end{proof}

\begin{lemma}
\label{le6}Let $\nu \in \mathfrak{A}_{sym}^{\alpha }$ satisfy \textbf{A}.

(i) Let $a\geq 0,\kappa >0,m=\left[ \kappa \right] +1$. For any $f\in
C_{b}^{\infty }\left( \mathbf{R}^{d}\right) ,$%
\begin{equation*}
\sup_{R\in (0,1],x}\left\vert \left( a-L^{\tilde{\nu}_{R}}\right) ^{\kappa
}D^{\gamma }f\left( x\right) \right\vert \leq C\left( 1+a\right) ^{\kappa
}\max_{\left\vert \mu \right\vert \leq \left\vert \gamma \right\vert
+2m}\left\vert D^{\mu }f\right\vert _{0}<\infty .
\end{equation*}%
If, in addition, for any multiindex $\gamma $, $\int \left\vert D^{\gamma
}f\left( x\right) \right\vert dx<\infty $, then%
\begin{equation*}
\sup_{R\in (0,1],x}\int \left\vert \left( a-L^{\tilde{\nu}_{R}}\right)
^{\kappa }D^{\gamma }f\left( x\right) \right\vert dx\leq C\left( 1+a\right)
^{\kappa }\max_{\left\vert \mu \right\vert \leq \left\vert \gamma
\right\vert +2m}\int \left\vert D^{\mu }f\left( x\right) \right\vert dx.
\end{equation*}%
The same holds with $\nu \in \mathfrak{A}^{\alpha }$ satisfying \textbf{A}
if $\kappa \in \mathbf{N}$.

(ii) Let $a>0,\kappa >0$. For any $f\in C_{b}^{\infty }\left( \mathbf{R}%
^{d}\right) ,$%
\begin{equation*}
\sup_{R\in (0,1],x}\left\vert \left( a-L^{\tilde{\nu}_{R}}\right) ^{-\kappa
}D^{\gamma }f\left( x\right) \right\vert \leq Ca^{-\kappa }\max_{\left\vert
\mu \right\vert \leq \left\vert \gamma \right\vert }\left\vert D^{\mu
}f\right\vert _{0}<\infty .
\end{equation*}%
If in addition, for any multiindex $\gamma $, $\int \left\vert D^{\gamma
}f\left( x\right) \right\vert dx<\infty $, then%
\begin{equation*}
\sup_{R\in (0,1]}\int \left\vert \left( a-L^{\tilde{\nu}_{R}}\right)
^{-\kappa }D^{\gamma }f\left( x\right) \right\vert dx\leq Ca^{-\kappa
}\max_{\left\vert \mu \right\vert \leq \left\vert \gamma \right\vert }\int
\left\vert D^{\mu }f\left( x\right) \right\vert dx.
\end{equation*}%
The same holds with $\nu \in \mathfrak{A}^{\alpha }$ satisfying \textbf{A}
if $\kappa \in \mathbf{N}$.
\end{lemma}

\begin{proof}
(i) Let $\kappa \in (0,1]$. Then%
\begin{eqnarray*}
&&\left( a-L^{\tilde{\nu}_{R}}\right) ^{\kappa }f\left( x\right) \\
&=&c_{\kappa }\int_{0}^{\infty }t^{-\kappa }\mathbf{E}\left[ e^{-at}f\left(
x+Z_{t}^{\tilde{\nu}_{R}}\right) -f\left( x\right) \right] \frac{dt}{t}%
=c_{\kappa }\int_{1}^{\infty }... \\
&&+c_{\kappa }\int_{0}^{1}t^{-\kappa }\int_{0}^{t}e^{-as}\mathbf{E}\left[
(L^{\tilde{\nu}_{R}}-a)f\left( x+Z_{s}^{\tilde{\nu}_{R}}\right) \right] ds%
\frac{dt}{t},%
\begin{array}{c}
x\in \mathbf{R}^{d}.%
\end{array}%
\end{eqnarray*}

By Lemma \ref{al2}, we have%
\begin{eqnarray*}
\sup_{R\in (0,1]}\int \left( \left\vert y\right\vert \wedge 1\right) \tilde{%
\nu}_{R}\left( dy\right)  &<&\infty \text{ if }\alpha \in \left( 0,1\right) ,
\\
\sup_{R\in (0,1]}\int \left( \left\vert y\right\vert ^{2}\wedge 1\right) 
\tilde{\nu}_{R}\left( dy\right)  &<&\infty \text{ if }\alpha =1, \\
\sup_{R\in (0,1]}\int \left( \left\vert y\right\vert ^{2}\wedge \left\vert
y\right\vert \right) \tilde{\nu}_{R}\left( dy\right)  &<&\infty \text{ if }%
\alpha \in (1,2),
\end{eqnarray*}%
and both inequalities easily follow. Applying them repeatedly we obtain the
claim for an arbitrary $\kappa >0.$

(ii) Indeed, for any $\kappa >0,a>0,$ and any multiindex $\gamma ,$%
\begin{equation*}
D^{\gamma }\left( a-L^{\tilde{\nu}_{R}}\right) ^{-\kappa }f\left( x\right)
=c_{\kappa }\int_{0}^{\infty }e^{-at}t^{\kappa }\mathbf{E}D^{\gamma }f\left(
x+Z_{t}^{\tilde{\nu}_{R}}\right) \frac{dt}{t},x\in \mathbf{R}^{d},
\end{equation*}%
and the claim obviously follows.
\end{proof}

\begin{lemma}
\label{le66}Let $\nu \in \mathfrak{A}^{\alpha }$ satisfy \textbf{A }and 
\textbf{B}. Let $g\in \mathcal{S}\left( \mathbf{R}^{d}\right) $ be such that 
$\hat{g}\in C_{0}^{\infty }\left( \mathbf{R}^{d}\right) ,0\notin $supp$%
\left( \hat{g}\right) $. Then there are constants $C,c$ so that%
\begin{equation*}
\sup_{R\in (0,1]}\int \left\vert \mathbf{E}g\left( x+Z_{t}^{\tilde{\nu}%
_{R}}\right) \right\vert dx\leq Ce^{-ct},t>0.
\end{equation*}
\end{lemma}

\begin{proof}
Let $F(t,x)=\mathbf{E}g\left( x+Z_{t}^{\tilde{\nu}_{R}}\right) ,x\in \mathbf{%
R}^{d},t>0$. We choose $\varepsilon >0$ so that supp$~\left( \hat{g}\right)
\subseteq \left\{ \xi :\left\vert \xi \right\vert \leq \varepsilon
^{-1}\right\} $. Let $\tilde{\nu}_{R,\varepsilon }\left( dy\right) =\chi
_{\left\{ \left\vert y\right\vert \leq \varepsilon \right\} }\tilde{\nu}%
_{R}\left( dy\right) ,R\in \left[ 0,1\right] $. Then for $\xi \in $supp$%
~\left( \hat{g}\right) $ and $\left\vert y\right\vert \leq \varepsilon ,$%
\begin{equation*}
1-\cos \left( \xi \cdot y\right) \geq \frac{1}{\pi }\left\vert \xi \cdot
y\right\vert ^{2}=\frac{\left\vert \xi \right\vert ^{2}}{\pi }\left\vert 
\hat{\xi}\cdot y\right\vert ^{2}
\end{equation*}%
with $\hat{\xi}=\xi /\left\vert \xi \right\vert .$ Therefore there is $%
c_{0}>0$ so that for any $\xi \in $supp$~\left( \hat{g}\right) $ and $R\in
(0,1],$ 
\begin{eqnarray}
-\func{Re}\psi ^{\tilde{\nu}_{R,\varepsilon }}\left( \xi \right) 
&=&\int_{\left\vert y\right\vert \leq \varepsilon }\left[ 1-\cos \left( \xi
\cdot y\right) \right] \tilde{\nu}_{R}\left( dy\right)   \label{f44} \\
&\geq &\frac{\left\vert \xi \right\vert ^{2}}{\pi }\int_{\left\vert
y\right\vert \leq \varepsilon }\left\vert \hat{\xi}\cdot y\right\vert ^{2}%
\tilde{\nu}_{R}\left( dy\right) \geq c_{0}\left\vert \xi \right\vert ^{2}. 
\notag
\end{eqnarray}%
Then%
\begin{equation*}
\hat{F}\left( t,\xi \right) =\exp \left\{ \psi ^{\tilde{\nu}_{R}}\left( \xi
\right) t\right\} \hat{g}\left( \xi \right) =\exp \left\{ \psi ^{\tilde{\nu}%
_{R,\varepsilon }}\left( \xi \right) t\right\} \exp \left\{ \psi \left( \xi
\right) t\right\} \hat{g}\left( \xi \right) ,\xi \in \mathbf{R}^{d},
\end{equation*}%
where $\exp \left\{ \psi \left( \xi \right) t\right\} $ is a characteristic
function of a probability distribution $P_{R,t}\left( dy\right) $ on $%
\mathbf{R}^{d}$. Hence%
\begin{equation*}
F\left( t,x\right) =\int H\left( t,x-y\right) P_{R,t}\left( dy\right) ,x\in 
\mathbf{R}^{d},
\end{equation*}%
with%
\begin{equation*}
H\left( t,x\right) =\mathcal{F}^{-1}\left[ \exp \left\{ \psi ^{\tilde{\nu}%
_{R,\varepsilon }}t\right\} \hat{g}\right] =\mathbf{E}g\left( x+Z_{t}^{%
\tilde{\nu}_{R,\varepsilon }}\right) ,x\in \mathbf{R}^{d}.
\end{equation*}%
Since%
\begin{equation*}
\int \left\vert F\left( t,x\right) \right\vert dx\leq \int \left\vert
H\left( t,x\right) \right\vert dx,
\end{equation*}%
it is enough to prove that%
\begin{equation}
\int \left\vert H\left( t,x\right) \right\vert dx\leq Ce^{-ct},t>0.
\label{f11}
\end{equation}%
Now, (\ref{f44}) implies that for any multiindex $\left\vert \gamma
\right\vert \leq n=[\frac{d}{2}]+3,$%
\begin{eqnarray*}
\int \left\vert x^{\gamma }H\left( t,x\right) \right\vert ^{2}dx &\leq
&C\int \left\vert D^{\gamma }\left[ \hat{g}\left( \xi \right) \exp \left\{
\psi ^{\tilde{\nu}_{R,\varepsilon }}\left( \xi \right) t\right\} \right]
\right\vert ^{2}d\xi  \\
&\leq &C_{1}e^{-c_{2}t},t>0.
\end{eqnarray*}%
Hence, denoting $d_{0}=\left[ \frac{d}{2}\right] +1,$%
\begin{eqnarray*}
\int \left\vert H\left( t,x\right) \right\vert dx &=&\int \left(
1+\left\vert x\right\vert ^{2}\right) ^{-d_{0}}\left\vert H\left( t,x\right)
\right\vert \left( 1+\left\vert x\right\vert ^{2}\right) ^{d_{0}}dx \\
&\leq &C\int \left\vert H\left( t,x\right) \right\vert ^{2}\left(
1+\left\vert x\right\vert ^{2}\right) ^{2d_{0}}dx\leq C_{1}e^{-c_{2}t},t>0.
\end{eqnarray*}%
Thus (\ref{f11}) follows, and 
\begin{equation}
\int \left\vert \mathbf{E}g\left( x+Z_{t}^{\tilde{\nu}_{R}}\right)
\right\vert dx\leq C_{1}e^{-c_{2}t},t>0.  \label{f8}
\end{equation}
\end{proof}

\begin{corollary}
\label{c1}Let $\nu \in \mathfrak{A}_{sym}^{\alpha }$ satisfy \textbf{A }and 
\textbf{B}. Let $g\in \mathcal{S}\left( \mathbf{R}^{d}\right) $ be such that 
$\hat{g}\in C_{0}^{\infty }\left( \mathbf{R}^{d}\right) ,0\notin $supp$%
\left( \hat{g}\right) $. Then for $a\geq 0,\kappa >0,R\in (0,1],$ 
\begin{eqnarray*}
\left( a-L^{\tilde{\nu}_{R}}\right) ^{-\kappa }g\left( x\right)  &=&-%
\mathcal{F}^{-1}\left[ \left( a-\psi ^{\tilde{\nu}_{R}}\right) ^{-\kappa }%
\hat{g}\right] \left( x\right)  \\
&=&c_{-\kappa }\int_{0}^{\infty }e^{-at}t^{\kappa }\mathbf{E}g\left(
x+Z_{t}^{\tilde{\nu}_{R}}\right) \frac{dt}{t},x\in \mathbf{R}^{d},
\end{eqnarray*}%
is $C_{b}^{\infty }$ -function, and for every multiindex $\gamma ,$ we have $%
D^{\gamma }\left( a-L^{\tilde{\nu}_{R}}\right) ^{-\kappa }g=\left( a-L^{%
\tilde{\nu}_{R}}\right) ^{-\kappa }D^{\gamma }g,R>0,$ and%
\begin{equation*}
\sup_{R\in (0,1],a\geq 0}\int \left\vert D^{\gamma }\left( a-L^{\tilde{\nu}%
_{R}}\right) ^{-\kappa }g\left( x\right) \right\vert ^{p}dx<\infty ,p\geq 1.
\end{equation*}

The same holds with $\nu \in \mathfrak{A}^{\alpha }$ satisfying \textbf{A }%
and \textbf{B} if $\kappa \in \mathbf{N}$.
\end{corollary}

\begin{proof}
Take $\eta \in C_{0}^{\infty }\left( \mathbf{R}^{d}\right) $ so that $\eta 
\hat{g}=\hat{g},0\notin $supp$~\left( \eta \right) $, and let $\tilde{\eta}=%
\mathcal{F}^{-1}\eta $. Let%
\begin{equation*}
F_{R}\left( t,x\right) =\mathbf{E}g\left( x+Z_{t}^{\tilde{\nu}_{R}}\right)
,t>0,x\in \mathbf{R}^{d}.
\end{equation*}%
Then%
\begin{equation*}
\hat{F}_{R}\left( t,\xi \right) =\exp \left\{ \psi ^{\tilde{\nu}_{R}}\left(
\xi \right) t\right\} \eta \left( \xi \right) \hat{g}\left( \xi \right) ,\xi
\in \mathbf{R}^{d},
\end{equation*}%
and%
\begin{equation*}
F_{R}\left( t,x\right) =\int H_{R}\left( t,x-y\right) g\left( y\right)
dy=\int g\left( x-y\right) H_{R}\left( t,y\right) dy,x\in \mathbf{R}^{d},
\end{equation*}%
with%
\begin{equation*}
H_{R}\left( t,x\right) =\mathcal{F}^{-1}\left[ \exp \left\{ \psi ^{\tilde{\nu%
}_{R}}t\right\} \eta \right] =\mathbf{E}\tilde{\eta}\left( x+Z_{t}^{\tilde{%
\nu}_{R}}\right) ,t>0,x\in \mathbf{R}^{d}.
\end{equation*}
By Lemma \ref{le66},%
\begin{equation*}
\sup_{R\in (0,1]}\int \left\vert H_{R}\left( t,y\right) \right\vert dy\leq
Ce^{-ct},t>0.
\end{equation*}%
Hence $F_{R}\left( t,\cdot \right) \in C_{b}^{\infty }\left( \mathbf{R}%
^{d}\right) ,t>0,$ and for each multiindex $\gamma $ and $p\geq 1,$%
\begin{equation*}
\sup_{R}\left( \int \left\vert D^{\gamma }F_{R}\left( t,x\right) \right\vert
^{p}dx\right) ^{1/p}\leq C\left( \int \left\vert D^{\gamma }g\left( x\right)
\right\vert ^{p}dx\right) ^{1/p}e^{-ct},t>0.
\end{equation*}
\end{proof}

\begin{corollary}
\label{c2}Let $\nu \in \mathfrak{A}^{\alpha }$ satisfy \textbf{A }and 
\textbf{B}. Let $g\in \mathcal{S}\left( \mathbf{R}^{d}\right) $ be such that 
$\hat{g}\in C_{0}^{\infty }\left( \mathbf{R}^{d}\right) ,0\notin $supp$%
\left( \hat{g}\right) $. Then there are constants $C,c>0$ so that%
\begin{equation*}
\sup_{R\in (0,1]}\int \left\vert \mathbf{E}L^{\tilde{\nu}_{R}}g\left(
x+Z_{t}^{\tilde{\nu}_{R}}\right) \right\vert dx\leq Ce^{-ct},t>0.
\end{equation*}
\end{corollary}

\begin{proof}
Let $h\in C_{0}^{\infty }\left( \mathbf{R}^{d}\right) ,0\leq h\leq 1,$ and $%
h\left( \xi \right) =1$ if $\xi \in $ supp $\left( g\right) $, $h\left( \xi
\right) =0$ in a neighborhood of zero. Let%
\begin{equation*}
G_{R}\left( t,x\right) =\mathbf{E}L^{\tilde{\nu}_{R}}g\left( x+Z_{t}^{\tilde{%
\nu}_{R}}\right) ,x\in \mathbf{R}^{d}\text{.}
\end{equation*}%
Then%
\begin{eqnarray*}
\hat{G}_{R}\left( t,\xi \right)  &=&\exp \left\{ \psi ^{\tilde{\nu}%
_{R}}\left( \xi \right) t\right\} \psi ^{\tilde{\nu}_{R}}\left( \xi \right) 
\hat{g}\left( \xi \right)  \\
&=&\exp \left\{ \psi ^{\tilde{\nu}_{R}}\left( \xi \right) t\right\} h\left(
\xi \right) \psi ^{\tilde{\nu}_{R}}\left( \xi \right) \hat{g}\left( \xi
\right) ,\xi \in \mathbf{R}^{d}.
\end{eqnarray*}%
Hence%
\begin{equation*}
G_{R}\left( t,x\right) =\int H_{R}\left( t,x-y\right) B_{R}\left( y\right)
dy,x\in \mathbf{R}^{d},
\end{equation*}%
where%
\begin{equation*}
B_{R}\left( x\right) =L^{\tilde{\nu}_{R}}g\left( x\right) ,H_{R}\left(
t,x\right) =\mathbf{E}h\left( x+Z_{t}^{\tilde{\nu}_{R}}\right) x\in \mathbf{R%
}^{d}.
\end{equation*}%
Thus, by Lemma \ref{le66}%
\begin{eqnarray*}
\sup_{R\in (0,1]}\int |G_{R}\left( t,x\right) |dx &\leq &\sup_{R\in
(0,1]}\int \left\vert H_{R}\left( t,x\right) \right\vert dx\sup_{R\in
(0,1]}\int \left\vert B_{R}\left( x\right) \right\vert dx \\
&\leq &Ce^{-ct},t>0.
\end{eqnarray*}
\end{proof}

\begin{lemma}
\label{le7}Let $\nu \in \mathfrak{A}_{sym}^{\alpha }$ satisfy \textbf{A }and 
\textbf{B}. Then

(i) For each $\beta ,\kappa >0,$ there is $C>0$ so that%
\begin{eqnarray*}
\left\vert L^{\nu ;\kappa }u\right\vert _{\beta ,\infty } &\leq &C\left\vert
u\right\vert _{\beta +\kappa ,\infty },u\in C_{b}^{\infty }\left( \mathbf{R}%
^{d}\right) , \\
\left\vert \left( I-L^{\nu }\right) ^{\kappa }u\right\vert _{\beta ,\infty }
&\leq &C\left\vert u\right\vert _{\beta +\kappa ,\infty },u\in C_{b}^{\infty
}\left( \mathbf{R}^{d}\right) ,
\end{eqnarray*}

(ii) For each $0<\kappa <\beta ,$ there is $C>0$ so that%
\begin{eqnarray*}
\left\vert u\right\vert _{\beta ,\infty } &\leq &C\left[ \left\vert L^{\nu
;\kappa }u\right\vert _{\beta -\kappa ,\infty }+\left\vert u\right\vert _{0}%
\right] ,u\in C_{b}^{\infty }\left( \mathbf{R}^{d}\right) , \\
\left\vert u\right\vert _{\beta ,\infty } &\leq &C\left\vert \left( I-L^{\nu
}\right) ^{\kappa }u\right\vert _{\beta -\kappa ,\infty },u\in C_{b}^{\infty
}\left( \mathbf{R}^{d}\right) .
\end{eqnarray*}

The same holds with $\nu \in \mathfrak{A}^{\alpha }$ satisfying \textbf{A }%
and \textbf{B} if $\kappa \in \mathbf{N}$.
\end{lemma}

\begin{proof}
Let $u\in C_{b}^{\infty }\left( \mathbf{R}^{d}\right) ,\tilde{\phi},\tilde{%
\phi}_{0}\in C_{0}^{\infty }\left( \mathbf{R}^{d}\right) $, be such that $%
\tilde{\phi}\phi =\phi ,\tilde{\phi}_{0}\phi _{0}=\phi _{0}$, where $\phi
_{0}=\mathcal{F}^{-1}\varphi _{0},$ and $\phi ,\varphi _{0}$ are the
functions in the definition of the spaces.

(i) Let $r\in \left[ 0,1\right] $. Then 
\begin{eqnarray*}
\left( r-L^{\nu }\right) ^{\kappa }u\ast \varphi _{j} &=&\mathcal{F}^{-1}%
\left[ \left( r-\psi ^{\nu }\right) ^{\kappa }\tilde{\phi}\left( N^{-j}\cdot
\right) \hat{u}\phi \left( N^{-j}\cdot \right) \right]  \\
&=&\int H_{r}^{j}\left( x-y\right) u\ast \varphi _{j}\left( y\right) dy,x\in 
\mathbf{R}^{d},j\geq 1, \\
\left( r-L^{\nu }\right) ^{\kappa }u\ast \varphi _{0} &=&\mathcal{F}^{-1}%
\left[ \left( r-\psi ^{\nu }\right) ^{\kappa }\tilde{\phi}_{0}\hat{u}\phi
_{0}\right] =\int H_{r}^{0}\left( x-y\right) u\ast \varphi _{0}\left(
y\right) dy,x\in \mathbf{R}^{d},
\end{eqnarray*}%
where $H_{r}^{j}=\mathcal{F}^{-1}\left[ \left( r-\psi ^{\nu }\right)
^{\kappa }\tilde{\phi}\left( N^{-j}\cdot \right) \right] ,j\geq 1,H_{r}^{0}=%
\mathcal{F}^{-1}\left[ \left( r-\psi ^{\nu }\right) ^{\kappa }\tilde{\phi}%
_{0}\right] .$ Let%
\begin{equation*}
G_{j}=w\left( N^{-j}\right) ^{-\kappa }\mathcal{F}^{-1}\left[ \left(
rw\left( N^{-j}\right) -\psi ^{\tilde{\nu}_{N^{-j}}}\right) ^{\kappa }\tilde{%
\phi}\right] ,j\geq 1.
\end{equation*}%
Since%
\begin{equation*}
\left( r-\psi ^{\nu }\left( \xi \right) \right) ^{\kappa }\tilde{\phi}\left(
N^{-j}\xi \right) =w\left( N^{-j}\right) ^{-\kappa }[rw\left( N^{-j}\right)
-\psi ^{\tilde{\nu}_{N^{-j}}}\left( N^{-j}\xi \right) ]^{\kappa }\tilde{\phi}%
\left( N^{-j}\xi \right) ,\xi \in \mathbf{R}^{d},
\end{equation*}%
it follows by Lemma \ref{le6} that%
\begin{eqnarray*}
\int \left\vert H_{r}^{j}\left( x\right) \right\vert dx &=&\int \left\vert
G_{j}\left( x\right) \right\vert dx \\
&=&w\left( N^{-j}\right) ^{-\kappa }\int \left\vert \left( rw\left(
N^{-j}\right) -L^{\tilde{\nu}_{N^{-j}}}\right) ^{\kappa }\tilde{\varphi}%
\left( x\right) \right\vert dx \\
&\leq &Cw\left( N^{-j}\right) ^{-\kappa },j\geq 0.
\end{eqnarray*}

(ii) Let $0<\kappa <\beta ,r\in \left[ 0,1\right] $. Then for $j\geq 1,$ 
\begin{eqnarray*}
u\ast \varphi _{j} &=&\left( r-L^{\nu }\right) ^{\kappa }\left( r-L^{\nu
}\right) ^{-\kappa }u\ast \varphi _{j} \\
&=&\mathcal{F}^{-1}\left[ \left( r-\psi ^{\nu }\right) ^{-\kappa }\tilde{\phi%
}\left( N^{-j}\cdot \right) \left( r-\psi ^{\nu }\right) ^{\kappa }\hat{u}%
\phi \left( N^{-j}\cdot \right) \right]  \\
&=&\int H_{r}^{j}\left( x-y\right) \left( r-L^{\nu }\right) ^{\kappa }u\ast
\varphi _{j}\left( y\right) dy,x\in \mathbf{R}^{d},j\geq 1,
\end{eqnarray*}%
where 
\begin{equation*}
H_{r}^{j}=\mathcal{F}^{-1}\left[ \left( r-\psi ^{\nu }\right) ^{-\kappa }%
\tilde{\phi}\left( N^{-j}\cdot \right) \right] ,j\geq 1,r\geq 0.
\end{equation*}%
Let%
\begin{equation*}
G_{j}=w\left( N^{-j}\right) ^{\kappa }\mathcal{F}^{-1}\left[ \left( rw\left(
N^{-j}\right) -\psi ^{\tilde{\nu}_{N^{-j}}}\right) ^{-\kappa }\tilde{\phi}%
\right] ,j\geq 1.
\end{equation*}%
It follows by Corollary \ref{c1} that there is $C$ independent of $r\geq
0,j\geq 1,$ so that%
\begin{eqnarray*}
\int \left\vert H_{r}^{j}\left( x\right) \right\vert dx &=&\int \left\vert
G_{j}\left( x\right) \right\vert dx=w\left( N^{-j}\right) ^{\kappa }\int
\left\vert \left( rw\left( N^{-j}\right) -L^{\tilde{\nu}_{N^{-j}}}\right)
^{-\kappa }\tilde{\varphi}\left( x\right) \right\vert dx \\
&\leq &Cw\left( N^{-j}\right) ^{\kappa }.
\end{eqnarray*}%
On the other hand, 
\begin{equation*}
\left\vert u\ast \varphi _{0}\right\vert _{0}\leq C\left\vert u\right\vert
_{0}.
\end{equation*}%
The statement follows.
\end{proof}

For $\beta >0,\kappa >0$, we define the following norms:%
\begin{eqnarray*}
\left\vert u\right\vert _{\nu ,\kappa ,\beta } &=&\left\vert u\right\vert
_{0}+\left\vert L^{\nu ;\kappa }u\right\vert _{\beta ,\infty },u\in
C_{b}^{\infty }\left( \mathbf{R}^{d}\right) , \\
\left\vert \left\vert u\right\vert \right\vert _{\nu ;\kappa ,\beta }
&=&\left\vert \left( I-L^{\nu }\right) ^{\kappa }u\right\vert _{\beta
,\infty },u\in C_{b}^{\infty }\left( \mathbf{R}^{d}\right) ,
\end{eqnarray*}%
with $\nu $ satisfying \textbf{A} and \textbf{B}. An immediate consequence
of Lemma \ref{le7} is the following norm equivalence.

\begin{corollary}
\label{pr3}Let $\nu \in \mathfrak{A}_{sym}^{\alpha }$ be a L\'{e}vy measure
satisfying \textbf{A }and \textbf{B}, $\beta >0,\kappa >0.$ Then norms $%
\left\vert u\right\vert _{\nu ,\kappa ,\beta }$, $\left\Vert u\right\Vert
_{\nu ,\kappa ,\beta }$ and $\left\vert u\right\vert _{\beta +\kappa ,\infty
}$ are equivalent on $C_{b}^{\infty }\left( \mathbf{R}^{d}\right) $.

The same holds with $\nu \in \mathfrak{A}^{\alpha }$ satisfying \textbf{A }%
and \textbf{B} if $\kappa \in \mathbf{N}$.
\end{corollary}

\begin{proof}
Let $\beta ,\kappa >0$. By Lemma \ref{le7},%
\begin{eqnarray*}
\left\vert \left( I-L^{\nu }\right) ^{\kappa }u\right\vert _{\beta ,\infty }
&\leq &C\left\vert u\right\vert _{\beta +\kappa ,\infty }\leq C\left[
\left\vert L^{\nu ;\kappa }u\right\vert _{\beta ,\infty }+\left\vert
u\right\vert _{0}\right]  \\
&=&C\left\vert u\right\vert _{\nu ;\beta ,\kappa },u\in C_{b}^{\infty
}\left( \mathbf{R}^{d}\right) .
\end{eqnarray*}%
On the other hand, by Lemmas \ref{equiv} and \ref{le7},%
\begin{eqnarray*}
\left\vert u\right\vert _{0} &\leq &C\left\vert u\right\vert _{\beta ,\infty
}\leq C\left\vert u\right\vert _{\beta +\kappa ,\infty }\leq C\left\vert
\left( I-L^{\nu }\right) ^{\kappa }u\right\vert _{\beta ,\infty }, \\
\left\vert L^{\nu ;\kappa }u\right\vert _{\beta ,\infty } &\leq &C\left\vert
u\right\vert _{\beta +\kappa ,\infty }\leq C\left\vert \left( I-L^{\nu
}\right) ^{\kappa }u\right\vert _{\beta ,\infty },u\in C_{b}^{\infty }\left( 
\mathbf{R}^{d}\right) .
\end{eqnarray*}
\end{proof}

\begin{corollary}
\label{pr4}Let $\nu \in \mathfrak{A}_{sym}^{\alpha }$ and $\pi \in \mathfrak{%
A}^{\alpha }$ be a L\'{e}vy measure satisfying \textbf{A }and \textbf{B }%
such that $w_{\pi }\sim w_{\nu }$. Then for any $\kappa \in \mathbf{N},\beta
>0,$ there are constants $c,C>0$ so that 
\begin{equation*}
\left\vert \left( L^{\pi }\right) ^{\kappa }u\right\vert _{\beta ,\infty
}\leq C_{1}\left\vert u\right\vert _{\nu ;\kappa ,\beta }\leq C_{2}\left[
\left\vert \left( L^{\pi }\right) ^{\kappa }u\right\vert _{\beta ,\infty
}+\left\vert u\right\vert _{0}\right] ,u\in C_{b}^{\infty }\left( \mathbf{R}%
^{d}\right) .
\end{equation*}
\end{corollary}

\begin{proof}
Indeed, by Corollary \ref{pr3}, 
\begin{eqnarray*}
\left\vert \left( L^{\pi }\right) ^{\kappa }u\right\vert _{\beta ,\infty }
&\leq &C\left\vert u\right\vert _{\kappa +\beta ,\infty }\leq C\left\vert
u\right\vert _{\nu ;\kappa ,\beta } \\
&\leq &C\left\vert u\right\vert _{\kappa +\beta ,\infty }\leq C_{2}\left[
\left\vert \left( L^{\pi }\right) ^{\kappa }u\right\vert _{\beta ,\infty
}+\left\vert u\right\vert _{0}\right] ,u\in C_{b}^{\infty }\left( \mathbf{R}%
^{d}\right) .
\end{eqnarray*}
\end{proof}

\subsubsection{Extension of norm equivalence to $\tilde{C}_{\infty \infty }^{%
\protect\beta }\left( \mathbf{R}^{d}\right) $}

We extend the definition of $\left( a-L^{\nu }\right) ^{\kappa }$ and the
norm equivalence (see Corollary \ref{pr3} above) from $C_{b}^{\infty }\left( 
\mathbf{R}^{d}\right) $ to $\tilde{C}_{\infty \infty }^{\beta }\left( 
\mathbf{R}^{d}\right) $. We start with the following observation.

\begin{remark}
\label{pr5} Let $0<\beta ^{\prime }<\beta $. Then for each $\varepsilon >0$
there is a constant $C_{\varepsilon }>0$ so that 
\begin{equation*}
\left\vert u\right\vert _{\beta ^{\prime },\infty }\leq \varepsilon
\left\vert u\right\vert _{\beta ,\infty }+C_{\varepsilon }\left\vert
u\right\vert _{0},u\in \tilde{C}_{\infty \infty }^{\beta }\left( \mathbf{R}%
^{d}\right) .
\end{equation*}

Indeed, For each $\varepsilon >0$ there is $K>1$ so that $w\left(
N^{-j}\right) ^{\beta -\beta ^{\prime }}\leq \varepsilon $ if $j\geq K$.
Hence 
\begin{eqnarray*}
w\left( N^{-j}\right) ^{-\beta ^{\prime }}\left\vert u\ast \varphi
_{j}\right\vert _{0} &=&w\left( N^{-j}\right) ^{\beta -\beta ^{\prime
}}w\left( N^{-j}\right) ^{-\beta }\left\vert u\ast \varphi _{j}\right\vert
_{0} \\
&\leq &\varepsilon \left\vert u\right\vert _{\beta ,\infty }+\max_{k\leq K}
\left[ w\left( N^{-k}\right) ^{-\beta ^{\prime }}\left\vert u\ast \varphi
_{k}\right\vert _{0}\right]  \\
&\leq &\varepsilon \left\vert u\right\vert _{\beta ,\infty }+C_{\varepsilon
}\left\vert u\right\vert _{0}.
\end{eqnarray*}
\end{remark}

\begin{proposition}
\label{app} Let $\beta \in \left( 0,\infty \right) $, $u\in \tilde{C}%
_{\infty ,\infty }^{\beta }\left( \mathbf{R}^{d}\right) $. Then there exists
a sequence $u_{n}\in C_{b}^{\infty }\left( \mathbf{R}^{d}\right) ,$ such
that 
\begin{equation*}
\left\vert u\right\vert _{\beta ,\infty }\leq \liminf_{n}\left\vert
u_{n}\right\vert _{\beta ,\infty },\quad \left\vert u_{n}\right\vert _{\beta
,\infty }\leq C\left\vert u\right\vert _{\beta ,\infty }
\end{equation*}%
for some $C>0$ that only depends on $d,N$.

Moreover, for any $0<\beta ^{\prime }<\beta $, 
\begin{equation*}
\left\vert u_{n}-u\right\vert _{\beta ^{\prime },\infty }\rightarrow 0\text{
as }n\rightarrow \infty .
\end{equation*}
\end{proposition}

\begin{proof}
Set $u_{n}=\sum_{j=0}^{n}u\ast \varphi _{j},n\geq 1$. Obviously, $u_{n}\in
C_{b}^{\infty }\left( \mathbf{R}^{d}\right) ,n\geq 1,$ and by Lemma \ref%
{equiv}, $u=\sum_{j=0}^{\infty }u\ast \varphi _{j}$ is a bounded continuous
function. Since%
\begin{equation*}
\varphi _{k}=\sum_{l=-1}^{1}\varphi _{k+l}\ast \varphi _{k},k\geq 1,~\varphi
_{0}=\left( \varphi _{0}+\varphi _{1}\right) \ast \varphi _{0},
\end{equation*}%
we have for $n>1,$%
\begin{eqnarray*}
\left( u-u_{n}\right) \ast \varphi _{k} &=&0,~k<n, \\
\left( u-u_{n}\right) \ast \varphi _{k} &=&\left( u\ast \varphi _{k-1}+u\ast
\varphi _{k}+u\ast \varphi _{k+1}\right) \ast \varphi _{k},~k>n+1, \\
\left( u-u_{n}\right) \ast \varphi _{n} &=&\left( u\ast \varphi
_{n+1}\right) \ast \varphi _{n}, \\
\left( u-u_{n}\right) \ast \varphi _{n+1} &=&\left( u\ast \varphi
_{n+1}+u\ast \varphi _{n+2}\right) \ast \varphi _{n+1}.
\end{eqnarray*}%
Hence there is a constant $C\,$\ so that 
\begin{equation*}
\left\vert u_{n}\ast \varphi _{j}\right\vert _{0}\leq C\left\vert u\ast
\varphi _{j}\right\vert _{0},j\geq 0,n>1,
\end{equation*}%
and for $n>1,$%
\begin{equation*}
\sup_{j<n}w\left( N^{-j}\right) ^{-\beta }\left\vert u\ast \varphi
_{j}\right\vert _{0}=\sup_{j<n}w\left( N^{-j}\right) ^{-\beta }\left\vert
u_{n}\ast \varphi _{j}\right\vert _{0}\leq \left\vert u_{n}\right\vert
_{\beta ,\infty }.
\end{equation*}%
Thus%
\begin{equation*}
\left\vert u\right\vert _{\beta ,\infty }\leq \underline{\lim }%
_{n}\left\vert u_{n}\right\vert _{\beta ,\infty },
\end{equation*}%
and%
\begin{equation*}
\left\vert u_{n}\right\vert _{\beta ,\infty }\leq C\left\vert u\right\vert
_{\beta ,\infty },n>1.
\end{equation*}%
Now, by Remark \ref{pr5}, for each $\varepsilon >0,$ there is a constant $%
C_{\varepsilon }$ so that%
\begin{equation*}
\left\vert u_{n}-u\right\vert _{\beta ^{\prime },\infty }\leq \varepsilon
\left( \left\vert u_{n}\right\vert _{\beta ,\infty }+\left\vert u\right\vert
_{\beta ,\infty }\right) +C_{\varepsilon }\left\vert u_{n}-u\right\vert _{0}.
\end{equation*}%
Since by Lemma \ref{equiv}, $\left\vert u_{n}-u\right\vert _{0}\rightarrow 0$%
, the statement follows.
\end{proof}

Using the approximating sequence introduced in Proposition \ref{app}, we can
extend $L^{\nu ;\kappa }u,\left( I-L^{\nu }\right) ^{\kappa }u,0<\kappa
<\beta ,$ to all $u\in \tilde{C}_{\infty ,\infty }^{\beta }\left( \mathbf{R}%
^{d}\right) ,\beta >0$.

\begin{proposition}
\label{cont} Let $\nu $ be a L\'{e}vy measure satisfying \textbf{A }and 
\textbf{B}, $\beta >0$ and $u\in \tilde{C}_{\infty ,\infty }^{\beta }\left( 
\mathbf{R}^{d}\right) $. Let $u_{n}\in C_{b}^{\infty }\left( \mathbf{R}%
^{d}\right) $ be an approximating sequence of $u$ in Proposition \ref{app}.
Then for each $\kappa \in \left( 0,\beta \right) $ there are bounded
continuous functions, denoted $\left( I-L^{\nu }\right) ^{\kappa }u,L^{\nu
;\kappa }u\in \tilde{C}_{\infty \infty }^{\beta -\kappa },$ so that for any $%
0<\beta ^{\prime }<\beta -\kappa ,$%
\begin{equation*}
\left\vert L^{\nu ;\kappa }u_{n}-L^{\nu ;\kappa }u\right\vert _{\beta
^{\prime },\infty }+\left\vert \left( I-L^{\nu }\right) ^{\kappa }u-\left(
I-L^{\nu }\right) ^{\kappa }u_{n}\right\vert _{\beta ^{\prime },\infty
}\rightarrow 0
\end{equation*}%
as $n\rightarrow \infty $. Moreover, for each $\kappa \in \left( 0,\beta
\right) $ there is $C>0$ independent of $u\in \tilde{C}_{\infty ,\infty
}^{\beta }\left( \mathbf{R}^{d}\right) $ so that%
\begin{equation}
\left\vert L^{\nu ;\kappa }u\right\vert _{\beta -\kappa ,\infty }\leq
C\left\vert u\right\vert _{\beta ,\infty },\left\vert \left( I-L^{\nu
}\right) ^{\kappa }u\right\vert _{\beta -\kappa ,\infty }\leq C\left\vert
u\right\vert _{\beta ,\infty },  \label{f22}
\end{equation}%
and 
\begin{equation}
\left\vert u\right\vert _{\beta ,\infty }\leq C\left[ \left\vert L^{\nu
;\kappa }u\right\vert _{\beta -\kappa ,\infty }+\left\vert u\right\vert _{0}%
\right] ,\left\vert u\right\vert _{\beta ,\infty }\leq C\left\vert \left(
I-L^{\nu }\right) ^{\kappa }u\right\vert _{\beta -\kappa ,\infty }.
\label{f23}
\end{equation}
\end{proposition}

\begin{proof}
Let $u\in \tilde{C}_{\infty ,\infty }^{\beta }\left( \mathbf{R}^{d}\right) $%
. By Proposition \ref{app}, there is a a sequence $u_{n}\in C_{b}^{\infty
}\left( \mathbf{R}^{d}\right) $ such that 
\begin{equation*}
\left\vert u\right\vert _{\beta ,\infty }\leq \liminf_{n}\left\vert
u_{n}\right\vert _{\beta ,\infty },\quad \left\vert u_{n}\right\vert _{\beta
,\infty }\leq C\left\vert u\right\vert _{\beta ,\infty },n\geq 1,
\end{equation*}%
for some $C>0$ independent of $u$, and for any $\kappa <\beta ,\beta
^{\prime }\in \left( 0,\beta -\kappa \right) ,$ see Lemma \ref{le7} as well, 
\begin{eqnarray*}
&&\left\vert L^{\nu ;\kappa }u_{n}-L^{\nu ;\kappa }u_{m}\right\vert _{\beta
^{\prime },\infty }+\left\vert \left( I-L^{\nu }\right) ^{\kappa
}u_{n}-\left( I-L^{\nu }\right) ^{\kappa }u_{m}\right\vert _{\beta ^{\prime
},\infty } \\
&\leq &C\left\vert u_{n}-u_{m}\right\vert _{\beta ^{\prime }}\rightarrow 0%
\text{ as }n,m\rightarrow \infty .
\end{eqnarray*}%
Hence there are bounded continuous functions, denoted \thinspace $L^{\mu
;\kappa }u,\left( I-L^{\nu }\right) ^{\kappa }u,$ so that%
\begin{equation*}
\left\vert L^{\nu ;\kappa }u_{n}-L^{\nu ;\kappa }u\right\vert
_{0}+\left\vert \left( I-L^{\nu }\right) ^{\kappa }u_{n}-\left( I-L^{\nu
}\right) ^{\kappa }u\right\vert _{0}\rightarrow 0
\end{equation*}%
as $n\rightarrow \infty .$ Thus%
\begin{eqnarray*}
&&\left\vert L^{\nu ;\kappa }u_{n}\ast \varphi _{j}-L^{\nu ;\kappa }u\ast
\varphi _{j}\right\vert _{0} \\
&&+\left\vert \left( I-L^{\nu }\right) ^{\kappa }u_{n}\ast \varphi
_{j}-\left( I-L^{\nu }\right) ^{\kappa }u\ast \varphi _{j}\right\vert _{0} \\
&\rightarrow &0,\quad j\geq 0,
\end{eqnarray*}%
as $n\rightarrow \infty $. Now, for each $m>1,$ and $a=0,1,$%
\begin{eqnarray*}
&&\sup_{j\leq m}w\left( N^{-j}\right) ^{-\beta -\kappa }\left\vert \left(
a-L^{\nu }\right) ^{\kappa }u\ast \varphi _{j}\right\vert _{0} \\
&=&\lim_{n\rightarrow \infty }\sup_{j\leq m}w\left( N^{-j}\right) ^{-\beta
-\kappa }\left\vert \left( a-L^{\nu }\right) ^{\kappa }u_{n}\ast \varphi
_{j}\right\vert _{0} \\
&\leq &\sup_{n}\left\vert \left( a-L^{\nu }\right) ^{\kappa
}u_{n}\right\vert _{\beta -\kappa }\leq \sup_{n}C\left\vert u_{n}\right\vert
_{\beta ,\infty }\leq C\left\vert u\right\vert _{\beta ,\infty }.
\end{eqnarray*}%
Hence $\left( a-L^{\nu }\right) ^{\kappa }u\in \tilde{C}_{\infty \infty
}^{\beta -\kappa }\left( \mathbf{R}^{d}\right) ,a=0,1$, and (\ref{f22})
holds.

Now for every $j\geq 0,$ we have

\begin{eqnarray}
\left[ \left( a-L^{\nu }\right) ^{\kappa }u\right] \ast \varphi _{j}
&=&\lim_{n}\left[ \left( a-L^{\nu }\right) ^{\kappa }u_{n}\right] \ast
\varphi _{j}  \label{f24} \\
&=&\left( a-L^{\nu }\right) ^{\kappa }\left[ u\ast \varphi _{j}\right]  
\notag
\end{eqnarray}%
uniformly. By the definition of the approximation sequence (see proof of
Proposition \ref{app}),%
\begin{equation*}
\left\vert \left( a-L^{\nu }\right) ^{\kappa }\left[ u_{n}\ast \varphi _{j}%
\right] \right\vert _{0}\leq C\left\vert \left( a-L^{\nu }\right) ^{\kappa }%
\left[ u\ast \varphi _{j}\right] \right\vert _{0},j\geq 0.
\end{equation*}%
Hence%
\begin{eqnarray*}
\left\vert u\right\vert _{\beta ,\infty } &\leq &\liminf_{n}\left\vert
u_{n}\right\vert _{\beta ,\infty }\leq C\liminf_{n}\left\vert \left(
I-L^{\nu }\right) ^{\kappa }u_{n}\right\vert _{\beta -\kappa ,\infty } \\
&\leq &C\left\vert \left( I-L^{\nu }\right) ^{\kappa }\right\vert _{\beta
-\kappa ,\infty },
\end{eqnarray*}
and similarly,%
\begin{eqnarray*}
\left\vert u\right\vert _{\beta ,\infty } &\leq &\liminf_{n}\left\vert
u_{n}\right\vert _{\beta ,\infty }\leq C\liminf_{n}\left[ \left\vert L^{\nu
;\kappa }u_{n}\right\vert _{\beta -\kappa ,\infty }+\left\vert
u_{n}\right\vert _{0}\right]  \\
&\leq &C\left[ \left\vert L^{\nu ;\kappa }u\right\vert _{\beta -\kappa
,\infty }+\left\vert u\right\vert _{0}\right] .
\end{eqnarray*}%
The statement is proved.
\end{proof}

\begin{proposition}
\label{thm4}Let $\nu \in \mathfrak{A}_{sym}^{\alpha }$ be a L\'{e}vy measure
satisfying \textbf{A }and \textbf{B}, $\beta >0,\kappa >0.$ Then norms $%
\left\vert u\right\vert _{\nu ,\kappa ,\beta }$, $\left\Vert u\right\Vert
_{\nu ,\kappa ,\beta }$ and $\left\vert u\right\vert _{\beta +\kappa ,\infty
}$ are equivalent on $\tilde{C}_{\infty \infty }^{\beta +\kappa }\left( 
\mathbf{R}^{d}\right) $.
\end{proposition}

\begin{proof}
We show the equivalence by repeating proof of Corollary \ref{pr3} where the
equivalence of the same norms on $C_{b}^{\infty }\left( \mathbf{R}%
^{d}\right) $ was derived. Only instead of Lemma \ref{le7} we use
Proposition \ref{cont}.
\end{proof}

\section{Proof of main theorem}

We assume in this section that \textbf{A }, \textbf{B }and \textbf{C}
hold. First we solve the equation with smooth input functions.

\begin{proposition}
\label{thm1}Let $\nu \in \mathfrak{A}^{\alpha }$, $\beta \in \left( 0,1\right)
$, $\lambda \geq 0$. Assume that $f\left( t,x\right) \in C_{b}^{\infty }\left(
H_{T}\right) $. Then there is a unique solution $u\in C_{b}^{\infty }\left(
H_{T}\right) $ to 
\begin{eqnarray}
\partial _{t}u\left( t,x\right)  &=&L^{\nu }u\left( t,x\right) -\lambda
u\left( t,x\right) +f\left( t,x\right) ,  \label{eeq1} \\
u\left( 0,x\right)  &=&0,\qquad \left( t,x\right) \in \left[ 0,T\right]
\times \mathbf{R}^{d}.  \notag
\end{eqnarray}
\end{proposition}

\begin{proof}
\textsc{Existence. }Denote $F\left( r,Z_{r}^{\nu }\right) =e^{-\lambda
\left( r-s\right) }f\left( s,x+Z_{r}^{\nu }-Z_{s}^{\nu }\right) ,s\leq r\leq
t,$ and apply the It\^{o} formula to $F\left( r,Z_{r}^{\nu }\right) $ on $%
\left[ s,t\right] $. 
\begin{eqnarray*}
&&e^{-\lambda \left( t-s\right) }f\left( s,x+Z_{t}^{\nu }-Z_{s}^{\nu
}\right) -f\left( s,x\right)  \\
&=&-\lambda \int_{s}^{t}F\left( r,Z_{r}^{\nu }\right) dr+\int_{s}^{t}\int
\chi _{\alpha }\left( y\right) y\cdot \nabla F\left( r,Z_{r-}^{\nu }\right) 
\tilde{J}\left( dr,dy\right)  \\
&&+\int_{s}^{t}\int \left[ F\left( r,Z_{r-}^{\nu }+y\right) -F\left(
r,Z_{r-}^{\nu }\right) -\chi _{\alpha }\left( y\right) y\cdot \nabla F\left(
r,Z_{r-}^{\nu }\right) \right] J\left( dr,dy\right) .
\end{eqnarray*}%
Take expectation for both sides and use the stochastic Fubini theorem, 
\begin{eqnarray*}
&&e^{-\lambda \left( t-s\right) }\mathbf{E}f\left( s,x+Z_{t}^{\nu
}-Z_{s}^{\nu }\right) -f\left( s,x\right)  \\
&=&-\lambda \int_{s}^{t}e^{-\lambda \left( r-s\right) }\mathbf{E}f\left(
s,x+Z_{r}^{\nu }-Z_{s}^{\nu }\right) dr+\int_{s}^{t}L^{\nu }e^{-\lambda
\left( r-s\right) }\mathbf{E}f\left( s,x+Z_{r}^{\nu }-Z_{s}^{\nu }\right) dr.
\end{eqnarray*}%
Integrate both sides over $\left[ 0,t\right] $ with respect to $s$ and
obtain 
\begin{eqnarray*}
&&\int_{0}^{t}e^{-\lambda \left( t-s\right) }\mathbf{E}f\left(
s,x+Z_{t}^{\nu }-Z_{s}^{\nu }\right) ds-\int_{0}^{t}f\left( s,x\right) ds \\
&=&-\lambda \int_{0}^{t}\int_{0}^{r}e^{-\lambda \left( r-s\right) }\mathbf{E}%
f\left( s,x+Z_{r}^{\nu }-Z_{s}^{\nu }\right) dsdr \\
&&+\int_{0}^{t}L^{\nu }\int_{0}^{r}e^{-\lambda \left( r-s\right) }\mathbf{E}%
f\left( s,x+Z_{r}^{\nu }-Z_{s}^{\nu }\right) dsdr,
\end{eqnarray*}%
which shows $u\left( t,x\right) =\int_{0}^{t}e^{-\lambda \left( t-s\right) }%
\mathbf{E}f\left( s,x+Z_{t-s}^{\nu }\right) ds$ solves $\eqref{eeq1}$ in the
integral sense. Obviously, as a result of the dominated convergence theorem
and Fubini's theorem, $u\in C_{b}^{\infty }\left( H_{T}\right) $. And by the
equation, $u$ is continuously differentiable in $t$.

\textsc{Uniqueness. }Suppose there are two solutions $u_{1},u_{2}$ solving
the equation, then $u:=u_{1}-u_{2}$ solves 
\begin{eqnarray}
\partial _{t}u\left( t,x\right) &=&L^{\nu }u\left( t,x\right) -\lambda
u\left( t,x\right) ,  \label{uni} \\
u\left( 0,x\right) &=&0.  \notag
\end{eqnarray}

Fix any $t\in \left[ 0,T\right] $. Apply the It\^{o} formula to $v\left(
t-s,Z_{s}^{\nu }\right) :=e^{-\lambda s}u\left( t-s,x+Z_{s}^{\nu }\right) $, 
$0\leq s\leq t,$ over $\left[ 0,t\right] $ and take expectation for both
sides of the resulting identity, then 
\begin{equation*}
u\left( t,x\right) =-\mathbf{E}\int_{0}^{t}e^{-\lambda s}\left[ \left(
-\partial _{t}u-\lambda u+L^{\nu }u\right) \left( t-s,x+Z_{s-}^{\nu }\right) %
\right] ds=0.
\end{equation*}
\end{proof}

\subsection{H\"{o}lder estimates of the smooth solution}

First we derive the estimates of the solution corresponding to a smooth
input function.

\begin{proposition}
\label{thm11} Let $\nu \in \mathfrak{A}^{\alpha }$, $\beta >0$ and \textbf{A}-%
\textbf{C} hold. Let $u\in C_{b}^{\infty }\left( H_{T}\right) $ be the
unique solution $u$ to $\eqref{eeq1}$ with $f\in C_{b}^{\infty }\left(
H_{T}\right) $. Then%
\begin{eqnarray}
\left\vert u\right\vert _{\beta ,\infty } &\leq &C\rho _{\lambda }\left(
T\right) \left\vert f\right\vert _{\beta ,\infty },  \label{est5} \\
\left\vert u\right\vert _{1+\beta ,\infty } &\leq &C\left[ 1+\rho _{\lambda
}\left( T\right) \right] ~\left\vert f\right\vert _{\beta ,\infty }
\label{est1}
\end{eqnarray}%
and for any $\mu \in \lbrack 0,1]$, $t^{\prime }<t\leq T$,
\begin{eqnarray}
&&\left\vert u\left( t,\cdot \right) -u\left( t^{\prime },\cdot \right)
\right\vert _{\mu +\beta ,\infty }  \label{est2} \\
&\leq &C\left\{ \left( t-t^{\prime }\right) ^{1-\mu }+\left[ 1+\rho
_{\lambda }\left( T\right) \right] \left\vert t-t^{\prime }\right\vert
\right\} \left\vert f\right\vert _{\beta ,\infty },  \notag
\end{eqnarray}%
where $\rho _{\lambda }\left( T\right) =\frac{1}{\lambda }\wedge T.$
\end{proposition}

\begin{proof}
Since $f\in C_{b}^{\infty }\left( H_{T}\right) $, by Lemma \ref{equiv}, 
\begin{eqnarray*}
f\left( t,x\right)  &=&\left( f\left( t,\cdot \right) \ast \varphi
_{0}\left( \cdot \right) \right) \left( x\right) +\sum_{j=1}^{\infty }\left(
f\left( t,\cdot \right) \ast \varphi _{j}\left( \cdot \right) \right) \left(
x\right)  \\
&=&f_{0}\left( t,x\right) +\sum_{j=1}^{\infty }f_{j}\left( t,x\right)
,\left( t,x\right) \in H_{T}.
\end{eqnarray*}%
Accordingly, for $j\geq 0,$ 
\begin{equation*}
u_{j}\left( t,x\right) =u\left( t,x\right) \ast \varphi _{j}\left( x\right)
=\int_{0}^{t}e^{-\lambda \left( t-s\right) }\mathbf{E}f_{j}\left(
s,x+Z_{t-s}^{\nu }\right) ds,\left( t,x\right) \in H_{T},
\end{equation*}%
is the solution to $\eqref{eeq1}$ with input $f_{j}=f\ast \varphi _{j}$. In
terms of Fourier transform,%
\begin{eqnarray*}
&&\hat{u}_{j}\left( t,\xi \right)\\
  &=&\int_{0}^{t}\exp \left\{ -\left( \lambda
-\psi ^{\nu }\left( \xi \right) \right) \left( t-s\right) \right\} \hat{f}%
\left( s,\xi \right) \phi \left( N^{-j}\xi \right) ds \\
&=&\int_{0}^{t}e^{-\lambda \left( t-s\right) }\exp \left\{ \psi ^{\tilde{\nu}%
_{N^{-j}}}\left( N^{-j}\xi \right) w\left( N^{-j}\right) ^{-1}\left(
t-s\right) \right\} \tilde{\phi}\left( N^{-j}\xi \right) \hat{f}_{j}\left(
s,\xi \right) ds,~j\geq 1.
\end{eqnarray*}%
Denote $w_{j}=w\left( N^{-j}\right) ^{-1}$. Then for $j\geq 0,$%
\begin{eqnarray*}
&&u_{j}\left( t,x\right)  \\
&=&\int_{0}^{t}e^{-\lambda \left( t-s\right) }\int H^{j}\left(
t-s,x-y\right) f_{j}\left( s,y\right) dyds,t\in \left[ 0,T\right] ,x\in 
\mathbf{R}^{d},
\end{eqnarray*}%
with 
\begin{eqnarray*}
H^{j}\left( t,x\right)  &=&N^{jd}\mathbf{E}\tilde{\varphi}\left(
N^{j}x+Z_{w_{j}t}^{\tilde{\nu}_{N^{-j}}}\right) ,\left( t,x\right) \in
H_{T},j\geq 1, \\
H^{0}\left( t,x\right)  &=&\mathbf{E}\tilde{\varphi}_{0}\left( x+Z_{t}^{\nu
}\right) ,\left( t,x\right) \in H_{T}.
\end{eqnarray*}%
Hence%
\begin{equation}
\int \left\vert H^{j}\left( t,x\right) \right\vert dx=\int \left\vert
G^{j}\left( t,x\right) \right\vert dx,t>0,j\geq 0,  \label{f31}
\end{equation}%
with $G^{0}=H^{0}$ and%
\begin{equation*}
G^{j}\left( t,x\right) =\mathbf{E}\tilde{\varphi}\left( x+Z_{w_{j}t}^{\tilde{%
\nu}_{N^{-j}}}\right) ,\left( t,x\right) \in \mathbf{R}^{d},j\geq 1.
\end{equation*}%
First we estimate the solution itself. For $j\geq 1,$ by Lemma \ref{le66},%
\begin{eqnarray*}
\left\vert u_{j}\left( t,\cdot \right) \right\vert _{0} &\leq &\left\vert
f_{j}\right\vert _{0}\int_{0}^{t}e^{-\lambda \left( t-s\right) }\int
\left\vert G^{j}\left( t-s,x\right) \right\vert dxds \\
&\leq &\left\vert f_{j}\right\vert _{0}\int_{0}^{t}e^{-\lambda \left(
t-s\right) }e^{-cw_{j}\left( t-s\right) }ds\leq Cw_{j}^{-1}\left\vert
f_{j}\right\vert _{0}.
\end{eqnarray*}%
Directly,%
\begin{equation*}
\left\vert u_{0}\left( t,\cdot \right) \right\vert _{0}\leq \left\vert
f_{0}\right\vert _{0}\int_{0}^{t}e^{-\lambda \left( t-s\right) }ds\leq
\left( \frac{1}{\lambda }\wedge T\right) \left\vert f_{0}\right\vert _{0}.
\end{equation*}%
Hence%
\begin{equation*}
\left\vert u\right\vert _{1+\beta ,\infty }\leq C\left[ 1+\left( \frac{1}{%
\lambda }\wedge T\right) \right] \left\vert f\right\vert _{\beta ,\infty }.
\end{equation*}%
Now we estimate time differences. For fixed $0<t^{\prime }<t\leq T,j\geq 0,$%
\begin{eqnarray*}
&&u_{j}\left( t,x\right) -u_{j}\left( t^{\prime },x\right)  \\
&=&\int_{t^{\prime }}^{t}e^{-\lambda \left( t-s\right) }\int H^{j}\left(
t-s,x-y\right) f_{j}\left( s,y\right) dyds \\
&&+\left( e^{-\lambda \left( t-t^{\prime }\right) }-1\right)
\int_{0}^{t^{\prime }}e^{-\lambda \left( t^{\prime }-s\right) }\int
H^{j}\left( t-s,x-y\right) f_{j}\left( s,y\right) dyds \\
&&+\int_{0}^{t^{\prime }}e^{-\lambda \left( t^{\prime }-s\right) }\int \left[
H^{j}\left( t-s,x-y\right) -H^{j}\left( t^{\prime }-s,x-y\right) \right]
f_{j}\left( s,y\right) dyds \\
&=&A_{1}^{j}\left( x\right) +A_{2}^{j}\left( x\right) +A_{3}^{j}\left(
x\right) ,x\in \mathbf{R}^{d}.
\end{eqnarray*}%
First, by Lemma \ref{le66}, for $j\geq 1,$%
\begin{eqnarray*}
\left\vert A_{1}^{j}\right\vert _{0} &\leq &\int_{t^{\prime
}}^{t}e^{-\lambda \left( t-s\right) }\int \left\vert G^{j}\left(
t-s,y\right) \right\vert dyds\left\vert f_{j}\right\vert _{0} \\
&\leq &C\int_{t^{\prime }}^{t}e^{-\lambda \left( t-s\right)
}e^{-cw_{j}\left( t-s\right) }ds\left\vert f_{j}\right\vert _{0}\leq
C\int_{t^{\prime }}^{t}e^{-cw_{j}\left( t-s\right) }ds\left\vert
f_{j}\right\vert _{0} \\
&\leq &Cw_{j}^{-1}\left[ 1-e^{-cw_{j}\left( t-t^{\prime }\right) }\right]
\left\vert f_{j}\right\vert _{0}.
\end{eqnarray*}%
And
\begin{equation*}
\left\vert A_{1}^{0}\right\vert \leq C\left\vert f_{0}\right\vert
_{0}\int_{t^{\prime }}^{t}e^{-\lambda \left( t-s\right) }ds\leq C\left\vert
f_{j}\right\vert _{0}\left\vert t-t^{\prime }\right\vert .
\end{equation*}%
By (\ref{f31}) and Lemma \ref{le66}, for $j\geq 1,$%
\begin{eqnarray*}
&&\left\vert A_{2}^{j}\right\vert _{0} \\
&\leq &\left( 1-e^{-\lambda \left( t-t^{\prime }\right) }\right)
\int_{0}^{t^{\prime }}e^{-\lambda \left( t^{\prime }-s\right) }\int
\left\vert G^{j}\left( t-s,y\right) \right\vert dyds\left\vert
f_{j}\right\vert _{0} \\
&\leq &C\left( 1-e^{-\lambda \left( t-t^{\prime }\right) }\right)
\int_{0}^{t^{\prime }}e^{-\lambda \left( t^{\prime }-s\right)
}e^{-cw_{j}\left( t-s\right) }ds\left\vert f_{j}\right\vert _{0}.
\end{eqnarray*}%
Thus for $j\geq 1,$%
\begin{eqnarray}
&&\left\vert A_{2}^{j}\right\vert _{0}  \label{f34} \\
&\leq &C\left( 1-e^{-\lambda \left( t-t^{\prime }\right) }\right)
\int_{0}^{t^{\prime }}e^{-\lambda \left( t^{\prime }-s\right) }ds\left\vert
f_{j}\right\vert _{0}\leq C\left\vert f_{j}\right\vert _{0}\left\vert
t-t^{\prime }\right\vert ,  \notag
\end{eqnarray}%
in the mean time,%
\begin{equation}
\left\vert A_{2}^{j}\right\vert _{0}\leq C\left\vert f_{j}\right\vert
_{0}\int_{0}^{t^{\prime }}e^{-cw_{j}\left( t^{\prime }-s\right) }ds\leq
C\left\vert f_{j}\right\vert _{0}w_{j}^{-1}.  \label{f35}
\end{equation}%
For $j=0$,
\begin{eqnarray*}
\left\vert A_{2}^{0}\right\vert _{0} &\leq &C\left( 1-e^{-\lambda \left(
t-t^{\prime }\right) }\right) \int_{0}^{t^{\prime }}e^{-\lambda \left(
t^{\prime }-s\right) }ds\left\vert f_{j}\right\vert _{0} \\
&\leq &C\left\vert t-t^{\prime }\right\vert \lambda \int_{0}^{t^{\prime
}}e^{-\lambda \left( t^{\prime }-s\right) }ds\left\vert f_{j}\right\vert
_{0}\leq C\left\vert f_{j}\right\vert _{0}\left\vert t-t^{\prime
}\right\vert .
\end{eqnarray*}%
At last, for $j\geq 1$,
\begin{equation*}
\left\vert A_{3}^{j}\right\vert _{0}\leq \left\vert f_{j}\right\vert
_{0}\int_{0}^{t^{\prime }}\int \left\vert G^{j}\left( t-s,y\right)
-G^{j}\left( t^{\prime }-s,y\right) \right\vert dyds.
\end{equation*}%
Note for $s\leq t^{\prime },$%
\begin{eqnarray*}
&&G^{j}\left( t-s,y\right) -G^{j}\left( t^{\prime }-s,y\right)  \\
&=&\mathbf{E}\left[ \tilde{\varphi}\left( y+Z_{w_{j}\left( t-s\right) }^{%
\tilde{\nu}_{N^{-j}}}\right) -\tilde{\varphi}\left( y+Z_{w_{j}(t^{\prime
}-s)}^{\tilde{\nu}_{N^{-j}}}\right) \right]  \\
&=&\mathbf{E}\int_{w_{j}(t^{\prime }-s)}^{w_{j}(t-s)}L^{\tilde{\nu}_{N^{-j}}}%
\tilde{\varphi}\left( y+Z_{r}^{\tilde{\nu}_{N^{-j}}}\right) dr,
\end{eqnarray*}%
and by Corollary \ref{c2},%
\begin{eqnarray*}
&&\int \left\vert G^{j}\left( t-s,y\right) -G^{j}\left( t^{\prime
}-s,y\right) \right\vert dy \\
&\leq &C\int_{w_{j}(t^{\prime }-s)}^{w_{j}(t-s)}e^{-cr}dr\leq
Ce^{-cw_{j}\left( t^{\prime }-s\right) }\left[ 1-e^{-cw_{j}\left(
t-t^{\prime }\right) }\right].
\end{eqnarray*}%
Thus for $j\geq 1,$%
\begin{eqnarray*}
\left\vert A_{3}^{j}\right\vert _{0} &\leq &C\left\vert f_{j}\right\vert _{0}
\left[ 1-e^{-cw_{j}\left( t-t^{\prime }\right) }\right] \int_{0}^{t^{\prime
}}e^{-cw_{j}\left( t^{\prime }-s\right) }ds \\
&=&Cw_{j}^{-1}\left\vert f_{j}\right\vert _{0}\left[ 1-e^{-cw_{j}\left(
t-t^{\prime }\right) }\right] \left[ 1-e^{-cw_{j}t^{\prime }}\right]  \\
&\leq &C\left\vert f_{j}\right\vert _{0}w_{j}^{-1}\left( 1-e^{-cw_{j}\left(
t-t^{\prime }\right) }\right).
\end{eqnarray*}%
In addition,%
\begin{eqnarray*}
\left\vert A_{3}^{0}\right\vert _{0} &\leq &C\left\vert f_{0}\right\vert
_{0}\int_{0}^{t^{\prime }}e^{-\lambda \left( t^{\prime }-s\right)
}ds\left\vert t-t^{\prime }\right\vert  \\
&\leq &C\left( \frac{1}{\lambda }\wedge T\right) \left\vert f_{0}\right\vert
_{0}\left\vert t-t^{\prime }\right\vert .
\end{eqnarray*}%
Summarizing,%
\begin{equation*}
\left\vert u_{0}\left( t,\cdot \right) -u_{0}\left( t^{\prime },\cdot
\right) \right\vert _{0}\leq C\left[ 1+\left( \frac{1}{\lambda }\wedge
T\right) \right] \left\vert f_{0}\right\vert _{0}\left\vert t-t^{\prime
}\right\vert,
\end{equation*}%
and%
\begin{eqnarray*}
&&\left\vert u_{j}\left( t,\cdot \right) -u_{j}\left( t^{\prime },\cdot
\right) \right\vert _{0} \\
&\leq &C\left\vert f_{j}\right\vert _{0}\left[ \left( \left\vert t-t^{\prime
}\right\vert \wedge w_{j}^{-1}\right) +w_{j}^{-1}\left( 1-e^{-cw_{j}\left(
t-t^{\prime }\right) }\right) \right]  \\
&=&C\left\vert f_{j}\right\vert _{0}w_{j}^{-1}\left[ \left( \left\vert
t-t^{\prime }\right\vert w_{j}\right) \wedge 1+\left( 1-e^{-cw_{j}\left(
t-t^{\prime }\right) }\right) \right],
\end{eqnarray*}%
which leads to 
\begin{equation*}
\left\vert u_{j}\left( t,\cdot \right) -u_{j}\left( t^{\prime },\cdot
\right) \right\vert _{0}\leq Cw_{j}^{-\mu }\left( t-t^{\prime }\right)
^{1-\mu },~\mu \in \lbrack 0,1],~j\geq 1.
\end{equation*}%
Thus%
\begin{eqnarray*}
&&\left\vert u\left( t,\cdot \right) -u\left( t^{\prime },\cdot \right)
\right\vert _{\mu +\beta ,\infty } \\
&\leq &C\left\vert f\right\vert _{\beta ,\infty }\left\{ \left( t-t^{\prime
}\right) ^{1-\mu }+\left[ 1+\left( \frac{1}{\lambda }\wedge T\right) \right]
\left\vert t-t^{\prime }\right\vert \right\} 
\end{eqnarray*}
for any $\mu \in \lbrack 0,1]$. The statement is proved.
\end{proof}

\subsection{General H\"{o}lder inputs}

\textsc{Existence and Estimates. }Given $f\in \tilde{C}_{\infty ,\infty
}^{\beta }\left( H_{T}\right) $, by Proposition \ref{app}, we can find a
sequence of functions $f_{n}$ in $C_{b}^{\infty }\left( H_{T}\right) $ such
that 
\begin{equation*}
\left\vert f_{n}\right\vert _{\beta ,\infty }\leq C\left\vert f\right\vert
_{\beta ,\infty },\quad \left\vert f\right\vert _{\beta ,\infty }\leq
\liminf_{n}\left\vert f_{n}\right\vert _{\beta ,\infty },
\end{equation*}%
and for any $0<\beta ^{\prime }<\beta $, 
\begin{equation*}
\left\vert f_{n}-f\right\vert _{0}\leq C\left\vert f_{n}-f\right\vert
_{\beta ^{\prime },\infty }\rightarrow 0\text{ as }n\rightarrow \infty .
\end{equation*}%
According to Theorems \ref{thm1} and \ref{thm11}, for each $f_{n}\in
C_{b}^{\infty }\left( \mathbf{R}^{d}\right) $, there is a corresponding
solution $u_{n}\in C_{b}^{\infty }\left( H_{T}\right) :$%
\begin{equation}
u_{n}\left( t,x\right) =\int_{0}^{t}\left[ L^{\nu }u_{n}\left( r,x\right)
-\lambda u_{n}\left( r,x\right) +f_{n}\left( r,x\right) \right] dr,\left(
t,x\right) \in \left[ 0,T\right] \times \mathbf{R}^{d}.  \label{f33}
\end{equation}
By Theorem \ref{thm11}, 
\begin{eqnarray*}
&&\left\vert L^{\nu }u_{m}-L^{\nu }u_{n}\right\vert _{\beta ^{\prime
},\infty } \\
&\leq &C\left\vert L^{\mu }u_{m}-L^{\mu }u_{n}\right\vert _{\beta ^{\prime
},\infty }\leq C\left\vert u_{m}-u_{n}\right\vert _{1+\beta ^{\prime
},\infty } \\
&\leq &C\left\vert f_{m}-f_{n}\right\vert _{\beta ^{\prime },\infty
}\rightarrow 0,\text{ as }m,n\rightarrow \infty
\end{eqnarray*}%
for all $\beta ^{\prime }\in \left( 0,\beta \right) $, which by Lemma \ref%
{equiv} implies that 
\begin{equation*}
\left\vert u_{n}-u_{m}\right\vert _{0}+\left\vert L^{\nu }u_{m}-L^{\mu
}u_{n}\right\vert _{0}\rightarrow 0\text{ as }m,n\rightarrow \infty .
\end{equation*}

So, there is $u\in \tilde{C}_{\infty ,\infty }^{1+\beta ^{\prime }}\left( H_{T}\right) $ for
any $\beta ^{\prime }\in \left( 0,\beta \right) $ such that $\left\vert
u_{n}-u\right\vert _{1+\beta ^{\prime },\infty }\rightarrow 0$ as $%
n\rightarrow \infty $. Passing to the limit in (\ref{f33}) we see that (\ref%
{f33}) holds for $u$. Let $\beta ^{\prime }\in \left( 0,\beta \right) $ and $%
\beta -\beta ^{\prime }<q_{1}^{-1}$. Then%
\begin{equation*}
\left\vert L^{1+\beta ^{\prime }}u_{n}\right\vert _{\beta -\beta ^{\prime
},\infty }\leq C\left\vert u_{n}\right\vert _{1+\beta,\infty }\leq
C\left\vert f_{n}\right\vert _{\beta ,\infty }\leq C\left\vert f\right\vert
_{\beta,\infty }
\end{equation*}%
implies that%
\begin{equation*}
\left\vert L^{1+\beta ^{\prime }}u_{n}\left( t,x\right) -L^{1+\beta ^{\prime
}}u_{n}\left( t,y\right) \right\vert \leq C\left\vert f\right\vert _{\beta
,\infty }w\left( \left\vert x-y\right\vert \right) ^{\beta -\beta ^{\prime
}},~x,y\in \mathbf{R}^{d}.
\end{equation*}%
and passing to the limit we see that%
\begin{equation*}
\left\vert L^{1+\beta ^{\prime }}u\left( t,x\right) -L^{1+\beta ^{\prime
}}u\left( t,y\right) \right\vert \leq C\left\vert f\right\vert _{\beta
,\infty }w\left( \left\vert x-y\right\vert \right) ^{\beta -\beta ^{\prime
}},~x,y\in \mathbf{R}^{d}.
\end{equation*}%
Hence $L^{1+\beta ^{\prime }}u\in \tilde{C}_{\infty ,\infty }^{\beta -\beta
^{\prime }}\left( \mathbf{R}^{d}\right) ,$ i.e., $u\in \tilde{C}_{\infty ,\infty }^{1+\beta
}\left( H_{T}\right) $ and%
\begin{equation*}
\left\vert u\right\vert _{1+\beta ,\infty }\leq C\left\vert f\right\vert
_{\beta ,\infty }.
\end{equation*}%
The convergence of $u_{n}$ to $u$ implies easily other estimates.

\textsc{Uniqueness. }Suppose there are two solutions $u_{1},u_{2}\in \tilde{C%
}_{\infty ,\infty }^{1+\beta }\left( H_{T}\right) $ to (\ref{1'}), then $%
u:=u_{1}-u_{2}$ solves 
\begin{equation}
u\left( t,x\right) =\int_{0}^{t}\left[ L^{\nu }u\left( r,x\right) -\lambda
u\left( r,x\right) \right] dr,~\left( t,x\right) \in \left[ 0,T\right]
\times \mathbf{R}^{d}.  \label{f32}
\end{equation}%
Let $g\in C_{0}^{\infty }\left( \mathbf{R}^{d}\right) ,0\leq g\leq 1,\int
gdx=1$. For $\varepsilon >0$, set 
\begin{equation*}
u_{\varepsilon }\left( t,x\right) =\int u\left( t,y\right) g_{\varepsilon
}\left( x-y\right) dy=\int \upsilon \left( t,x-y\right) g_{\varepsilon
}\left( y\right) dy,~\left( t,x\right) \in H_{T},
\end{equation*}%
with $g_{\varepsilon }\left( x\right) =\varepsilon ^{-d}g\left(
x/\varepsilon \right) ,x\in \mathbf{R}^{d}$. Then $u_{\varepsilon }\in 
\tilde{C}_{b}^{\infty }\left( H_{T}\right) $ solves (\ref{f32}). Hence $%
u_{\varepsilon }=0$ for all $\varepsilon >0.$ Thus $u=0$, the solution is
unique.

\section{Appendix}

We simply state a few results that were used in this paper. Let $\nu \in 
\mathfrak{A}^{\alpha }$, and 
\begin{eqnarray*}
\delta \left( r\right)  &=&\delta _{\nu }\left( r\right) =\nu \left( \left\{
\left\vert y\right\vert >r\right\} \right) >0,r>0, \\
w &=&w_{\nu }\left( r\right) =\delta \left( r\right)
^{-1},r>0,\lim_{r\rightarrow 0}w\left( r\right) =0.
\end{eqnarray*}%
We assume that $w=w_{\nu }$ is an O-RV function at zero, i.e., 
\begin{equation*}
r_{1}\left( \varepsilon \right) =\overline{\lim_{x\rightarrow 0}}\frac{%
\delta \left( \varepsilon x\right) ^{-1}}{\delta \left( x\right) ^{-1}}%
<\infty ,\varepsilon >0.
\end{equation*}

By Theorem 2 in \cite{aa}, the following limits exist:%
\begin{equation}
p_{1}=\lim_{\varepsilon \rightarrow 0}\frac{\log r_{1}\left( \varepsilon
\right) }{\log \varepsilon }\leq q_{1}=\lim_{\varepsilon \rightarrow \infty }%
\frac{\log r_{1}\left( \varepsilon \right) }{\log \varepsilon }.  \label{af1}
\end{equation}

\begin{lemma}
\label{al1}Assume $w=w_{\nu }$ is an O-RV function at zero.

a) Let $\beta >0$ and $\tau >-\beta p_{1}.$ There is $C>0$ so that%
\begin{equation*}
\int_{0}^{x}t^{\tau }w\left( t\right) ^{\beta }\frac{dt}{t}\leq Cx^{\tau
}w\left( x\right) ^{\beta },x\in (0,1],
\end{equation*}%
and $\lim_{x\rightarrow 0}x^{\tau }w\left( x\right) ^{\beta }=0.$

b) Let $\beta >0$ and $\tau <-\beta q_{1}$. There is $C>0$ so that%
\begin{equation*}
\int_{x}^{1}t^{\tau }w\left( t\right) ^{\beta }\frac{dt}{t}\leq Cx^{\tau
}w\left( x\right) ^{\beta },x\in (0,1],\text{ }
\end{equation*}%
and $\lim_{x\rightarrow 0}x^{\tau }w\left( x\right) ^{\beta }=\infty .$

c) Let $\beta <0$ and $\tau >-\beta q_{1}$. There is $C>0$ so that 
\begin{equation*}
\int_{0}^{x}t^{\tau }w\left( t\right) ^{\beta }\frac{dt}{t}\leq Cx^{\tau
}w\left( x\right) ^{\beta },x\in (0,1],
\end{equation*}%
and $\lim_{x\rightarrow 0}x^{\tau }w\left( x\right) ^{\beta }=0.$

d) Let $\beta <0$ and $\tau <-\beta p_{1}$. There is $C>0$ so that 
\begin{equation*}
\int_{x}^{1}t^{\tau }w\left( t\right) ^{\beta }\frac{dt}{t}%
=\int_{1}^{x^{-1}}t^{-\tau }w\left( \frac{1}{t}\right) ^{\beta }\frac{dt}{t}%
\leq Cx^{\tau }w\left( x\right) ^{\beta },x\in (0,1],
\end{equation*}%
and $\lim_{x\rightarrow 0}x^{\tau }w\left( x\right) ^{\beta }=\infty .$
\end{lemma}

\begin{proof}
The claims follow easily by Theorems 3, 4 in \cite{aa}. Because of the
similarities, we will prove c) only. Let $\beta <0$ and $\tau >-\beta q_{1}$%
. Then 
\begin{eqnarray*}
\overline{\lim_{t\rightarrow \infty }}\frac{w\left( \frac{1}{\varepsilon t}%
\right) ^{\beta }}{w\left( \frac{1}{t}\right) ^{\beta }} &=&\overline{%
\lim_{x\rightarrow 0}}\frac{w\left( x\right) ^{-\beta }}{w\left( \varepsilon
^{-1}x\right) ^{-\beta }}=\overline{\lim_{x\rightarrow 0}}\frac{w\left(
\varepsilon \varepsilon ^{-1}x\right) ^{-\beta }}{w\left( \varepsilon
^{-1}x\right) ^{-\beta }} \\
&=&\overline{\lim_{x\rightarrow 0}}\frac{w\left( \varepsilon x\right)
^{-\beta }}{w\left( x\right) ^{-\beta }}=r_{1}\left( \varepsilon \right)
^{-\beta }<\infty ,\varepsilon >0.
\end{eqnarray*}%
Hence $w\left( \frac{1}{t}\right) ^{\beta },t\geq 1$, is an O-RV function at
infinity with%
\begin{equation*}
p=\lim_{\varepsilon \rightarrow 0}\frac{\log r_{1}\left( \varepsilon \right)
^{-\beta }}{\log \varepsilon }=-\beta p_{1}\leq -\beta
q_{1}=\lim_{\varepsilon \rightarrow \infty }\frac{\log r_{1}\left(
\varepsilon \right) ^{-\beta }}{\log \varepsilon }=q.
\end{equation*}

Then for $x\in (0,1],$%
\begin{equation*}
\int_{0}^{x}t^{\tau }w\left( t\right) ^{\beta }\frac{dt}{t}%
=\int_{x^{-1}}^{\infty }t^{-\tau }w\left( \frac{1}{t}\right) ^{\beta }\frac{%
dt}{t}\leq Cx^{\tau }w\left( x\right) ^{\beta }
\end{equation*}%
by Theorem 3 in \cite{aa}, and $\lim_{x\rightarrow 0}x^{\tau }w\left(
x\right) ^{\beta }=0$ according to Theorem 4 in \cite{aa}.
\end{proof}

\begin{corollary}
\label{ac1}Assume $w=w_{\nu }$ is an O-RV function at zero and $p_{1}>0$.
Let $N>1$, $\beta >0$. Then%
\begin{equation*}
\sum_{j=0}^{\infty }w\left( N^{-j}\right) ^{\beta }<\infty .
\end{equation*}
\end{corollary}

\begin{proof}
Indeed,%
\begin{equation*}
\sum_{j=0}^{\infty }w\left( N^{-j}\right) ^{\beta }\leq \int_{0}^{\infty
}w\left( N^{-x}\right) ^{\beta }dx\leq C\int_{0}^{1}w\left( t\right) ^{\beta
}\frac{dt}{t}<\infty ,
\end{equation*}%
because, by Lemma \ref{al1}a)$,$%
\begin{equation*}
\int_{0}^{x}w\left( t\right) ^{\beta }\frac{dt}{t}\leq Cw\left( x\right)
^{\beta },x\in \left[ 0,1\right] .
\end{equation*}
\end{proof}

We will need some L\'{e}vy measure moment estimates.

\begin{lemma}
\label{al2}Let $\nu \in \mathfrak{A}^{\alpha },$ and \thinspace $w=w_{\nu }$
be an O-RV function at zero with $p_{1},q_{1}$ defined in (\ref{af1}). Assume 
\begin{eqnarray*}
0 &<&p_{1}\leq q_{1}<1\text{ if }\alpha \in \left( 0,1\right) , \\
1 &\leq &p_{1}\leq q_{1}<2\text{ if }\alpha =1, \\
1 &<&p_{1}\leq q_{1}<2\text{ if }\alpha \in \left( 1,2\right) .
\end{eqnarray*}

Then

(i) $\ $%
\begin{eqnarray*}
\sup_{R\in (0,1]}\int \left( \left\vert y\right\vert \wedge 1\right) \tilde{%
\nu}_{R}\left( dy\right)  &<&\infty \text{ if }\alpha \in \left( 0,1\right) ,
\\
\sup_{R\in (0,1]}\int \left( \left\vert y\right\vert ^{2}\wedge 1\right) 
\tilde{\nu}_{R}\left( dy\right)  &<&\infty \text{ if }\alpha =1, \\
\sup_{R\in (0,1]}\int \left( \left\vert y\right\vert ^{2}\wedge \left\vert
y\right\vert \right) \tilde{\nu}_{R}\left( dy\right)  &<&\infty \text{ if }%
\alpha \in (1,2).
\end{eqnarray*}

(ii)%
\begin{equation*}
\inf_{R\in (0,1]}\int_{\left\vert y\right\vert \leq 1}\left\vert
y\right\vert ^{2}\tilde{\nu}_{R}\left( dy\right) \geq c_{1},
\end{equation*}%
for some $c_{1}>0.$
\end{lemma}

\begin{proof}
(i) Let $\alpha \in \left( 0,1\right) .$ Then by Lemma \ref{al1},%
\begin{eqnarray*}
\int_{\left\vert y\right\vert \leq 1}\left\vert y\right\vert \tilde{\nu}%
_{R}\left( dy\right)  &=&R^{-1}\int_{\left\vert y\right\vert \leq
R}\left\vert y\right\vert \nu \left( dy\right)  \\
&=&R^{-1}\int_{0}^{R}[\delta \left( s\right) -\delta \left( R\right) ]ds,
\end{eqnarray*}%
and 
\begin{equation*}
\int_{\left\vert y\right\vert \leq 1}\left( \left\vert y\right\vert \wedge
1\right) \tilde{\nu}_{R}\left( dy\right) =R^{-1}\int_{0}^{R}w\left( s\right)
^{-1}ds\leq C,~R\in (0,1].
\end{equation*}

Let $\alpha =1.$ Then, using Lemma \ref{al1} we have%
\begin{equation*}
\int_{\left\vert y\right\vert \leq 1}\left( \left\vert y\right\vert
^{2}\wedge 1\right) \tilde{\nu}_{R}\left( dy\right)
=2R^{-2}\int_{0}^{R}s^{2}w\left( s\right) ^{-1}\frac{ds}{s}\leq C,~R\in (0,1].
\end{equation*}%
Let $\alpha \in \left( 1,2\right)$. Then similarly,%
\begin{eqnarray*}
&&R^{-1}\int_{\left\vert y\right\vert >R}\left\vert y\right\vert \nu \left(
dy\right) =R^{-1}\int_{0}^{\infty }\delta \left( s\vee R\right) ds \\
&=&\delta \left( R\right) +\int_{R}^{\infty }\delta \left( s\right)
ds=\delta \left( R\right) +\int_{R}^{\infty }w\left( s\right) ^{-1}ds
\end{eqnarray*}%
and with $R\in (0,1],$%
\begin{eqnarray}
R^{-2}\int_{\left\vert y\right\vert \leq R}\left\vert y\right\vert ^{2}\nu
\left( dy\right)  &=&2R^{-2}\int_{0}^{R}s^{2}[w\left( s\right) ^{-1}-w\left(
R\right) ^{-1}]\frac{ds}{s}  \label{af3} \\
&=&2R^{-2}\int_{0}^{R}s^{2}w\left( s\right) ^{-1}\frac{ds}{s}-w\left(
R\right) ^{-1}.  \notag
\end{eqnarray}%
Hence, by Lemma \ref{al1},%
\begin{eqnarray*}
&&\int \left( \left\vert y\right\vert ^{2}\wedge \left\vert y\right\vert
\right) \tilde{\nu}_{R}\left( dy\right)  \\
&\leq &2R^{-2}\int_{0}^{R}s^{2}w\left( s\right) ^{-1}\frac{ds}{s}%
+\int_{R}^{1}w\left( s\right) ^{-1}ds+\int_{1}^{\infty }w\left( s\right)
^{-1}ds \\
&=&2R^{-2}\int_{0}^{R}s^{2}w\left( s\right) ^{-1}\frac{ds}{s}%
+\int_{R}^{1}w\left( s\right) ^{-1}ds+\int_{\left\vert y\right\vert
>1}\left\vert y\right\vert \nu \left( dy\right)  \\
&\leq &Cw\left( R\right) ^{-1},~R\in (0,1].
\end{eqnarray*}

(ii) By (\ref{af3}), for $R\in (0,1],$ 
\begin{eqnarray*}
\int_{\left\vert y\right\vert \leq 1}\left\vert y\right\vert ^{2}\tilde{\nu}%
_{R}\left( dy\right) &=&w\left( R\right) \int_{\left\vert y\right\vert \leq
1}\left\vert y\right\vert ^{2}\nu _{R}\left( dy\right) \\
&=&2R^{-2}\int_{0}^{R}s^{2}[\frac{w\left( R\right) }{w\left( s\right) }-1]%
\frac{ds}{s}=2\int_{0}^{1}s^{2}[\frac{w\left( R\right) }{w\left( Rs\right) }%
-1]\frac{ds}{s}.
\end{eqnarray*}%
Hence, by Fatou's lemma,%
\begin{equation*}
\underline{\lim }_{R\rightarrow 0}\int_{\left\vert y\right\vert \leq
1}\left\vert y\right\vert ^{2}\tilde{\nu}_{R}\left( dy\right) \geq
2\int_{0}^{1}s^{2}[\frac{1}{r_{1}\left( s\right) }-1]\frac{ds}{s}=c_{1}>0
\end{equation*}%
if $\left\vert \left\{ s\in \left[ 0,1\right] :r_{1}\left( s\right)
<1\right\} \right\vert >0$, and 
\begin{equation*}
\lim \inf_{R\rightarrow 0}\frac{w\left( R\right) }{w\left( Rs\right) }=\frac{%
1}{\lim \sup_{R\rightarrow 0}\frac{w\left( Rs\right) }{w\left( R\right) }}=%
\frac{1}{r_{1}\left( s\right) },~s\in (0,1].
\end{equation*}
\end{proof}

According to \cite{dm}, Chapter 3, 70-74, any L\'{e}vy measure $\nu \in 
\mathfrak{A}^{\alpha }$ can be disintegrated as%
\begin{equation*}
\nu \left( \Gamma \right) =-\int_{0}^{\infty }\int_{S_{d-1}}\chi _{\Gamma
}\left( rw\right) \Pi \left( r,dw\right) d\delta \left( r\right) ,\Gamma \in 
\mathcal{B}\left( \mathbf{R}_{0}^{d}\right) ,
\end{equation*}%
where $\delta =\delta _{\nu }$, and $\Pi \left( r,dw\right) ,r>0,$ is a
measurable family of measures on the unit sphere $S_{d-1}$ with $\Pi \left(
r,S_{d-1}\right) =1,r>0.$ The following is a straightforward consequence of
Lemma \ref{al2}(ii).

\begin{corollary}
\label{ac2}Let $\nu \in \mathfrak{A}^{\alpha },$ 
\begin{equation*}
\nu \left( \Gamma \right) =-\int_{0}^{\infty }\int_{S_{d-1}}\chi _{\Gamma
}\left( rw\right) \Pi \left( r,dw\right) d\delta \left( r\right) ,\Gamma \in 
\mathcal{B}\left( \mathbf{R}_{0}^{d}\right) ,
\end{equation*}%
where $\delta =\delta _{\pi },\Pi \left( r,dw\right),~r>0,$ is a measurable
family of measures on $S_{d-1}$ with $\Pi \left( r,S_{d-1}\right) =1,~r>0.$
Assume $w=w_{\nu }=\delta_{\nu } ^{-1}$ be an O-RV function at zero satisfying assumptions
of Lemma \ref{al2}, and 
\begin{equation}
\inf_{\left\vert \hat{\xi}\right\vert =1}\int_{S_{d-1}}\left\vert \hat{\xi}%
\cdot w\right\vert ^{2}\Pi \left( r,dw\right) \geq c_{0}>0.  \label{af2}
\end{equation}

Then assumption \textbf{B }holds.
\end{corollary}

\begin{proof}
Indeed, for $\left\vert \hat{\xi}\right\vert =1,R\in (0,1],$ with $C>0,$%
\begin{eqnarray*}
&&\int_{\left\vert y\right\vert \leq 1}\left\vert \hat{\xi}\cdot
y\right\vert ^{2}\nu _{R}\left( dy\right) \\
&=&R^{-2}\int_{\left\vert y\right\vert \leq R}\left\vert \hat{\xi}\cdot
y\right\vert ^{2}\nu \left( dy\right)
=-R^{-2}\int_{0}^{R}\int_{S_{d-1}}\left\vert \hat{\xi}\cdot w\right\vert
^{2}\Pi \left( r,dw\right) r^{2}d\delta \left( r\right) \\
&\geq &-c_{0}R^{-2}\int_{0}^{R}r^{2}d\delta \left( r\right)
=c_{0}R^{-2}\int_{\left\vert y\right\vert \leq R}\left\vert y\right\vert
^{2}\nu \left( dy\right) =c_{0}\int_{\left\vert y\right\vert \leq
1}\left\vert y\right\vert ^{2}\nu _{R}\left( dy\right) .
\end{eqnarray*}%
Hence by Lemma \ref{al2}(ii), 
\begin{eqnarray*}
\inf_{R\in (0,1]}\inf_{\left\vert \hat{\xi}\right\vert =1}\int_{\left\vert
y\right\vert \leq 1}\left\vert \hat{\xi}\cdot y\right\vert ^{2}\tilde{\nu}%
_{R}\left( dy\right) &\geq &c_{0}~\inf_{R\in (0,1]}\int_{\left\vert
y\right\vert \leq 1}\left\vert y\right\vert ^{2}\tilde{\nu}_{R}\left(
dy\right) \\
&\geq &c_{0}c_{1}>0.
\end{eqnarray*}
\end{proof}

\begin{remark}
\label{ar1}Let $\alpha \in \left( 0,2\right) ,$ $\nu \in \mathfrak{A}%
^{\alpha }$, and $w_{\nu }$ be an O-RV function at zero, $p_{1}>0.$ By Theorems 3 and
4 in \cite{aa}, for any $\sigma \in (0,p_{1}),$%
\begin{eqnarray*}
\int_{r<\left\vert y\right\vert \leq 1}\left\vert y\right\vert ^{\sigma }\nu
\left( dy\right)  &=&\sigma \int_{r}^{1}t^{\sigma }w\left( t\right) ^{-1}%
\frac{dt}{t}-\delta \left( 1\right)  \\
&\geq &cr^{\sigma }w\left( r\right) ^{-1}-\delta \left( 1\right) \rightarrow
\infty 
\end{eqnarray*}%
as $r\rightarrow 0$. Hence $p_{1}\leq \alpha $. On the other hand for any $%
\sigma >q_{1}$, by Lemma \ref{al2}, 
\begin{equation*}
\int_{0<\left\vert y\right\vert \leq 1}\left\vert y\right\vert ^{\sigma }\nu
\left( dy\right) \leq \sigma \int_{0}^{1}t^{\sigma }w\left( t\right) ^{-1}%
\frac{dt}{t}<\infty ,
\end{equation*}%
and $\alpha \leq q_{1}.$
\end{remark}

\end{document}